\documentclass[english]{msjproc}
\usepackage{lmodern}

\usepackage{etex}
\usepackage{amsmath,amsfonts,amssymb,amsthm,amscd,mathrsfs,mathtools}
\usepackage{shuffle}
\usepackage{tikz}
\usepackage{enumitem,pifont,bm}
\usepackage{xcolor}
\usepackage[all]{xy}
\usepackage[title,titletoc,toc]{appendix}

\setlist{nolistsep}
\AtBeginDocument{
      \setlength\abovedisplayskip{2pt}
      \setlength\belowdisplayskip{2pt}}
\makeatletter

\newcommand{\ZZ}{\mathbb{Z}}
\newcommand{\NN}{\mathbb{N}}
\newcommand{\kk}{\Bbbk}
\newcommand{\Map}{\operatorname{Map}}
\newcommand{\Hom}{\operatorname{Hom}}

\newcommand{\Id}{\operatorname{Id}}

\renewcommand{\ge}{\geqslant}
\renewcommand{\le}{\leqslant}
\newcommand{\osigma}{\boldsymbol{\sigma}}
\newcommand{\osigmas}{\osigma^S}
\newcommand{\osigman}{\osigma^N}
\newcommand{\ow}{\overline{w}}
\newcommand{\oDelta}{\overline{\Delta}}
\newcommand{\emptyw}{\boldsymbol{\varepsilon}}

\newcommand{\sh}{\underset{\sigma}{\shuffle}}
\newcommand{\Sh}{\underset{-\sigma}{\shuffle}}
\renewcommand{\cshuffle}{\ensuremath{\rotatebox[origin=c]{180}{$\shuffle$}}}
\newcommand{\csh}{\underset{\sigma}{\cshuffle}}
\newcommand{\Csh}{\underset{-\sigma}{\cshuffle}}

\newcommand{\Norm}{\mathbf{Norm}}
\newcommand{\SG}{\mathbf{SG}}
\newcommand{\Mon}{\mathbf{M}}
\newcommand{\oSG}{\overline{\SG}}
\newcommand{\oMon}{\overline{\Mon}}
\newcommand{\oNorm}{\overline{\Norm}}
\newcommand{\TotSh}{\mathcal{Q}\mathcal{S}}
\newcommand{\oTotSh}{\overline{\TotSh}}

\newcommand\longmapsfrom{\mathrel{\reflectbox{\ensuremath{\longmapsto}}}}

\newcommand*\rcircled[1]{$\,$\tikz[baseline=(char.base)]{
            \node[shape=circle,draw,inner sep=1pt] (char) {#1};}}

\definecolor{choco}{rgb}{.525,.27,.075}
\definecolor{mygreen}{rgb}{0,.455,0} 
\definecolor{myviolet}{rgb}{.45,.05,.545}
\definecolor{myred}{rgb}{.545,0,0}
\definecolor{myblue}{rgb}{.024,.15,.645} 
\colorlet{MyRed}{myred!40!white}
\colorlet{MyBlue}{myblue!40!white}
\colorlet{MyGreen}{mygreen!40!white}
\colorlet{MyViolet}{myviolet!40!white}
\definecolor{MyGrey}{rgb}{.804,.804,.756} 
\usepackage{caption}
\usepackage[colorlinks=false,linkbordercolor=MyGrey,citebordercolor=MyGrey,
urlbordercolor=MyGrey, pdfauthor={V. Lebed}, pdftitle={Plactic monoids}]{hyperref}
\usepackage[all]{hypcap}

  \theoremstyle{plain}
\newtheorem*{MR}{Main result}  
\newtheorem{theorem}{Theorem}

\newtheorem{proposition}{Proposition}[section]

\newtheorem{lemma}[proposition]{Lemma}
  \theoremstyle{remark}

\newtheorem{remark}[proposition]{Remark}
  \theoremstyle{definition}
\newtheorem{definition}[proposition]{Definition}
\newtheorem{notation}[proposition]{Notation}
\newtheorem{example}[proposition]{Example}

\begin{document}
  \title{Cohomology of idempotent braidings, \\ with applications to factorizable monoids}
  \author{Victoria LEBED}{University of Nantes}
  \email{lebed.victoria@gmail.com}

  \maketitle

  \begin{abstract}  
We develop new methods for computing the Hochschild (co)homology of monoids which can be presented as the structure monoids of idempotent set-theoretic solutions to the Yang--Baxter equation. These include free and symmetric monoids; factorizable monoids, for which we find a generalization of the K\"{u}nneth formula for direct products; and plactic monoids. Our key result is an identification of the (co)homologies in question with those of the underlying YBE solutions, via the explicit quantum symmetrizer map. This partially answers questions of Farinati--Garc{\'{\i}}a-Galofre and Dilian Yang. We also obtain new structural results on the (co)homology of general YBE solutions.
  \end{abstract}

  \section{Introduction}

The {Yang--Baxter equation}~\eqref{E:YBE} plays a fundamental role in mathematical areas ranging from statistical mechanics to quantum field theory, from low-dimensional topology to quantum group theory. Attention to its set-theoretic solutions, called \emph{braidings}, dates back to Drinfel$'$d~\cite{DrST}. They cover an important part of the algebraic diversity of general solutions, while being more manageable.

Our original results mainly concern {idempotent braidings}. Of little interest in physics or topology, they do become useful in algebra. In particular, they provide a powerful unifying tool, simultaneously treating very different algebraic structures:
\begin{enumerate}
\item free and free commutative monoids;
\item factorizable monoids;
\item distributive lattices;
\item Young tableaux and plactic monoids.
\end{enumerate}
The first three are addressed here, the last one in a follow-up paper~\cite{LebedPlactic}. The third one generalizes to bounded Garside families (including Garside monoids); this is reserved for a separate paper as well. 
 One more reason to focus on idempotent braidings is the associated representations of {Coxeter monoids}\footnote{Also known as \emph{$0$-Hecke monoids}, they were defined and studied for all Coxeter groups. Since only the symmetric group case is relevant for us, we use simplified terms and notations.}. These monoids appeared in the work of Tsaranov~\cite{Tsaranov}, and since then were applied to Hecke algebras, to the Bruhat order on Coxeter groups, to Tits buildings, and to planar graphs \cite{FomGr,HST0Hecke,DolanTrimble,GanMaz,Kenney,Kenney2}. Another interesting feature of idempotent braidings is their interpretation inside the normalization paradigm of Dehornoy--Guiraud \cite{DehGui}. All these aspects are presented in Section~\ref{S:IdempotentBraiding}.

The first (co)homology theory for general braidings is due to Carter, Elham\-dadi, and Saito \cite{HomologyYB}. It was further developed and extended to solutions in any preadditive monoidal category by the author~\cite{Lebed1}. The motivation of~\cite{HomologyYB} was to generalize the powerful knot and knotted surface invariants constructed out of rack cocycles. Racks, and more generally self-distributive structures, yield a fundamental example of braidings. The goal of~\cite{Lebed1} was to unify the (co)homology theories of basic algebraic structures (associative and Lie algebras, racks, bialgebras, Hopf (bi)modules etc.) by interpreting them as YBE solutions. Quite surprisingly, the two approaches resulted in the same theory. It inherited a handy graphical calculus from knot theory, and a great level of generality from its unifying mission. Section~\ref{S:BrHom} is a reminder on these {braided (co)homology theories}. They have been extensively studied for various particular cases of braidings; see for example \cite{RackHom,FRS_BirackHom,BirackHom,PrzYBHom,LebedVendramin,LebedVendramin2,BikeiHom} and references therein. However, the idempotent case has until now remained unexplored.

Braided cohomology groups with suitable coefficients carry additional structure, which captures more information on the braiding. For instance, Farinati and Garc{\'{\i}}a-Galofre \cite{FarinatiGalofre} described a cup product for cohomology with trivial coefficients. In Section~\ref{S:BrHomCup} we extend it to more general coefficients. Further, we lift it to the cochain level, where in the case of trivial coefficients it is graded commutative up to a homotopy, explicitly described in Section~\ref{S:BrHomCircle}. The graded commutativity in cohomology follows. As an example, we show the group cohomology with trivial coefficients and its classical cup product to be a particular case of our braided constructions.

To any braiding one classically associates a monoid. This gives a rich source of quadratic groups and algebras with interesting properties, widely exploited in \cite{GIBergh,Rump,JesOknI,Chouraqui,DehYBE}. On the other hand, these monoids allow group-theoretic approaches to the classification of braidings \cite{ESS,Soloviev,LuYanZhu}. Farinati and Garc{\'{\i}}a-Galofre \cite{FarinatiGalofre} observed that the quantum symmetrizer $\TotSh$ connects the (co)homology of a braiding with that of its monoid. They showed $\TotSh$ to yield an isomorphism for involutive braidings in characteristic zero, and asked if this was true for other types of braidings. The same question was independently put by Dilian Yang \cite[Question 7.5]{YangGraphYBE}. We obtain a positive answer for idempotent braidings:
\begin{MR}
The quantum symmetrizer yields a quasi-isomorphism between a certain quotient of the braided chain complex for an idempotent braiding, and the Hochschild chain complex for its monoid, with the same coefficients. An analogous result holds in cohomology. For suitable coefficients, $\TotSh$ respects the cup products.
\end{MR}
See Section~\ref{S:BrHomSG} for a precise statement. In practice, braided complexes are considerably smaller than those coming from the bar resolution for the associated monoid. We thus get an efficient tool for computing Hochschild (co)homology. We apply it to free and free commutative monoids, where it recovers classical small resolutions; and to factorizable monoids, where it gives a generalization of the K\"{u}nneth formula. In~\cite{LebedPlactic}, our result yields efficient resolutions of plactic monoids, advancing the (co)homology computations of Lopatkin~\cite{Lopatkin}.
 
\medskip
\textbf{Acknowledgments.}
The author thanks Patrick Dehornoy for bringing her attention to the mysterious appearance of braidings in Lopatkin's work on plactic monoids, which was the starting point of this project; and Friedrich Wagemann for fruitful discussions on cup products over a cup of coffee. The author was supported by the program ANR-11-LABX-0020-01 and Henri Lebesgue Center (University of Nantes). 

\section{Idempotent braidings}\label{S:IdempotentBraiding}

A \emph{braided set} is a set~$X$ endowed with a \emph{braiding}\footnote{In the Introduction, we used the term \emph{braiding} for a braided set for brevity.}, i.e., a (non-invertible) solution $\sigma \colon X^{\times 2} \to X^{\times 2}$ to the \emph{Yang--Baxter equation} (= \emph{YBE})
\begin{equation}\label{E:YBE}
(\sigma \times \Id)  (\Id \times \sigma)  (\sigma \times \Id) = (\Id \times \sigma)  (\sigma \times \Id)  (\Id \times \sigma)
\end{equation}
on $X^{\times 3}$. An \emph{idempotent} braiding obeys the additional axiom $\sigma  \sigma = \sigma$. A braiding induces an action of the \emph{positive braid monoid} \begin{equation}\label{E:Bn}
B_k^+ = \langle \, b_1, \ldots, b_{k-1} \, | \, b_ib_j=b_jb_i \text{ for } |i-j| >1,  \, b_ib_{i+1}b_i = b_{i+1}b_ib_{i+1} \,\rangle^+
\end{equation}
on~$X^{\times k}$, for all $k \in \NN$, via
\begin{align*}
b_i &\mapsto \Id^{\times (i-1)} \times \sigma \times \Id^{\times (k-i-1)}.
\end{align*}
In the idempotent case, this action descends to the quotient
\begin{equation}\label{E:Cn}
C_k = \langle \, b_1, \ldots, b_{k-1} \, | \, b_ib_j=b_jb_i \text{ for } |i-j| >1,  \, b_ib_{i+1}b_i = b_{i+1}b_ib_{i+1}, \, b_ib_i=b_i \,\rangle^+
\end{equation}
of~$B_k^+$, referred to as \emph{Coxeter monoid}.

The \emph{graphical calculus} is extensively used in what follows, rendering our constructions more intuitive. Braided diagrams represent maps between sets, a set being associated to each strand; horizontal glueing corresponds to Cartesian product, vertical glueing to composition (read from bottom to top), vertical lines to identity maps, crossings to braidings, and trivalent vertices to (co)products. With these conventions, the YBE becomes the diagram from Figure~\ref{P:YBE}\rcircled{A}, which is the braid- and knot-theoretic Reidemeister~$\mathrm{III}$ move. A more advanced example is the classical extension~$\osigma$ of a braiding~$\sigma$ from~$X$ to the set~$X^*$ of words on the alphabet~$X$: its most concise definition is graphical (Figure~\ref{P:YBE}\rcircled{B}).
    
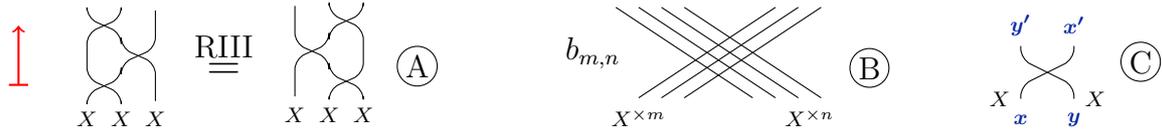
\begin{figure}[!h]\centering
\begin{tikzpicture}[xscale=0.45,yscale=0.4]
\draw [|->, red, thick]  (-2,0.5) -- (-2,2.5);
\draw [rounded corners] (0,0)--(0,0.25)--(1,0.75)--(1,1.25)--(2,1.75)--(2,3);
\draw [rounded corners] (1,0)--(1,0.25)--(0,0.75)--(0,2.25)--(1,2.75)--(1,3);
\draw [rounded corners] (2,0)--(2,1.25)--(1,1.75)--(1,2.25)--(0,2.75)--(0,3);
\node  at (0,0) [below] {$\scriptstyle X$};
\node  at (1,0) [below] {$\scriptstyle X$};
\node  at (2,0) [below] {$\scriptstyle X$};
\node  at (4,1.5){\Large $\overset{\mathrm{RIII}}{=}$};
\end{tikzpicture}
\begin{tikzpicture}[xscale=0.45,yscale=0.4]
\draw [rounded corners] (1,1)--(1,1.25)--(2,1.75)--(2,3.25)--(1,3.75)--(1,4);
\draw [rounded corners] (0,1)--(0,2.25)--(1,2.75)--(1,3.25)--(2,3.75)--(2,4);
\draw [rounded corners] (2,1)--(2,1.25)--(1,1.75)--(1,2.25)--(0,2.75)--(0,4);
\node  at (0,1) [below] {$\scriptstyle X$};
\node  at (1,1) [below] {$\scriptstyle X$};
\node  at (2,1) [below] {$\scriptstyle X$};
\node  at (3.5,2){\rcircled{A}};
\node  at (7,0){ };
\end{tikzpicture}
\begin{tikzpicture}[xscale=0.3,yscale=0.2]
\draw (1,0)--(7,6);
\draw (2,0)--(8,6);
\draw (3,0)--(9,6);
\draw (6,0)--(0,6);
\draw (7,0)--(1,6);
\draw (8,0)--(2,6);
\draw (9,0)--(3,6);
\node [below] at (1,0) {$\scriptstyle X^{\times m}$};
\node [below] at (8.5,0) {$\scriptstyle X^{\times n}$};
\node  at (-1,3){$ b_{m,n}$};
\node  at (11,2){\rcircled{B}};
\node  at (15,0){ };
\end{tikzpicture}
\begin{tikzpicture}[scale=0.7]
\draw [rounded corners] (0,0)--(0,0.25)--(1,0.75)--(1,1);
\draw [rounded corners] (1,0)--(1,0.25)--(0,0.75)--(0,1);
\node  at (0,-0.4) [myblue]  {$\scriptstyle \boldsymbol{x}$};
\node  at (1,-0.4) [myblue] {$\scriptstyle \boldsymbol{y}$};
\node  at (0,1.4) [myblue] {$\scriptstyle \boldsymbol{y'}$};
\node  at (1,1.4) [myblue] {$\scriptstyle \boldsymbol{x'}$};
\node  at (0,0) [left] {$\scriptstyle X$};
\node  at (1,0) [right] {$\scriptstyle X$};
\node  at (2.2,0.7){\rcircled{C}};
\end{tikzpicture}
\caption{A graphical version of the YBE; extension of a braiding from~$X$ to~$X^*$; color propagation through a crossing}\label{P:YBE}
\end{figure}

A diagram consisting exclusively of crossings also represents an element of~$B_n^+$ or~$C_n$; for instance, Figure~\ref{P:YBE}\rcircled{A} depicts the relation  $b_1b_{2}b_1 = b_{2}b_1b_{2}$ in~$B_3^+$, and Figure~\ref{P:YBE}\rcircled{B} an element $b_{m,n} \in B_{m+n}^+$. Note that a product $ab$ in~$B_n^+$ or~$C_n$ corresponds to a diagram representing~$a$ placed on top of that representing~$b$.

Associating \emph{colors} (i.e., arbitrary elements of the corresponding sets) to the bottom free ends of a diagram and applying to them the map encoded by the diagram, one determines the top colors. Figure~\ref{P:YBE}\rcircled{C} contains a simple case of this process, referred to as \textit{color propagation}: here the top colors are $(y',x') = \sigma (x,y)$. 

We now cite the key properties of Coxeter monoids. They involve the \emph{symmetric groups}~$S_k$, with their classical presentation
$$S_k = \langle \, s_1, \ldots, s_{k-1} \, | \, s_is_j=s_js_i \text{ for } |i-j| >1,  \, s_is_{i+1}s_i = s_{i+1}s_is_{i+1}, \, s_is_i= 1 \,\rangle.$$

\begin{lemma}[\cite{Tsaranov}]\label{L:CoxMon}
\begin{enumerate}
\item A set-theoretic bijection between~$C_k$ and~$S_k$ is established by sending any $b \in C_k$ with a shortest representation $b_{i_1}\cdots b_{i_n}$ to $s_{i_1}\cdots s_{i_n}$, which is in its turn a shortest representation of a word in~$S_k$.
\item The longest element
\begin{align}\label{E:Delta}
\Delta_k &= b_1 (b_2b_1) \cdots (b_{k-1} \cdots b_2b_1)
\end{align}
of~$C_k$ (Figure~\ref{P:Delta}) absorbs any $b \in C_k$:
\begin{align}\label{E:Absorb}
\Delta_k b = b \Delta_k &= \Delta_k.
\end{align}
\end{enumerate}
\end{lemma}

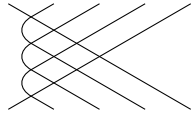
\begin{figure}[!h]\centering
\begin{tikzpicture}[xscale=0.6,yscale=0.7]
 \draw [rounded corners=10] (-1,0.5) -- (-5,2.5);
 \draw [rounded corners=10] (-2,0.5) -- (-5,2) -- (-4,2.5); 
 \draw [rounded corners=10] (-3,0.5) -- (-5,1.5) -- (-3,2.5);
 \draw [rounded corners=10] (-4,0.5) -- (-5,1) -- (-2,2.5);   
 \draw [rounded corners=10] (-5,0.5) -- (-1,2.5);  
\end{tikzpicture} 
\caption{The longest element~$\Delta_k$ of~$C_k$}\label{P:Delta}
\end{figure}

To a braided set it is classical to associate a certain semigroup, which captures its basic algebraic properties. For an idempotent braiding, we interpret it from the normalization perspective.
 
\begin{definition}\label{D:StrSemigroup}
 The \emph{structure semigroup} of a braided set $(X,\sigma)$ is given by the following presentation:
$$\SG(X,\sigma) = \langle \, X \; | \;\, xy = y'x' \text{ whenever } \sigma(x,y)=(y',x'), \, x,y \in X \,\rangle^+.$$
The \emph{structure monoid} $\Mon(X,\sigma)$ of $(X,\sigma)$ is the monoid given by the same presentation. 
The set of \emph{$\sigma$-normal} words is defined as
$$ \Norm(X,\sigma) = \{ \, x_{1}\ldots x_{k} \in X^* \;\, | \;\, \forall 1 \le j < k, \; \sigma (x_{j},x_{{j+1}}) = (x_{j},x_{{j+1}}) \,\}.$$
Notation $\Norm^{+}(X,\sigma)$ is used when the empty word is excluded. 
A representative $x_{1}\ldots x_{k}$ of an element of $\SG(X,\sigma)$, with $x_{j} \in X$, is called its \emph{normal form} if it is a $\sigma$-normal word.
\end{definition}

\begin{proposition}\label{PR:DeltaNormal}
 Take a set~$X$ with an idempotent braiding~$\sigma$. The action of the elements $\Delta_k \in C_k$ on $X^{\times k}$ via~$\sigma$ induces bijections~\footnote{We use somewhat abusive notations.}
\begin{align*}
\Delta_* \colon \SG(X,\sigma) &\overset{1:1}{\longrightarrow} \Norm^{+}(X,\sigma),\\
\Delta_* \colon \Mon(X,\sigma) &\overset{1:1}{\longrightarrow} \Norm(X,\sigma).
\end{align*}  
Further, any $\boldsymbol{w} \in \Mon(X,\sigma)$ has a unique normal form, given by $\Delta_*(\boldsymbol{w})$.
\end{proposition}

In the idempotent case, one can thus freely switch between structure monoids and $\sigma$-normal words. In the context of rewriting systems, a similar result was obtained by Dehornoy and Guiraud \cite[Proposition 5.1.1]{DehGui}. Stated in their terms, our~$\Delta_*$ yields a quadratic normalization of type $(3,3)$.

\begin{proof}
Take an element $\boldsymbol{w}$ of $\SG(X,\sigma)$, written as $x_{1}\ldots x_{k}$ in the basis~$X$. Put $(x''_{k}, \ldots, x''_{1}) = \Delta_k (x_{1}, \ldots, x_{k})$. The relation $b_{k-j} \Delta_k = \Delta_k$ in~$C_n$ (Lemma~\ref{L:CoxMon}) implies $\sigma (x''_{{j+1}},x''_{{j}}) = (x''_{{j+1}},x''_{{j}})$. Hence the word $x''_{k}\ldots x''_{1}$ is $\sigma$-normal, and the image of~$\Delta_k$ lies in $\Norm^{+} (X,\sigma)$. Since $\Delta_k$ is a product of some $b_i$s, $x''_{k}\ldots x''_{1}$ also represents $\boldsymbol{w}$, and is thus its normal form. For the same reason, $\Delta_k$ restricts to the identity on $\Norm^{+} (X,\sigma)$. Further, the value of~$\Delta_k$ on $(x_{1},\ldots, x_{k})$ does not change when subsequent elements $(x_{j},x_{{j+1}})$ are replaced with $(x'_{{j+1}},x'_{{j}}) = \sigma (x_{j},x_{{j+1}})$: this follows from the relation $\Delta_k b_j = \Delta_k$ in~$C_n$ (again Lemma~\ref{L:CoxMon}). Summarizing, $\Delta_k$ associates to all words representing some $\boldsymbol{w} \in \SG(X,\sigma)$ the unique $\sigma$-normal word representing~$\boldsymbol{w}$.
\end{proof}

\begin{remark}
The proposition is to be compared with a similar result for an involutive~$\sigma$ and the associated $S_k$-actions. Namely, for a field~$\kk$ of characteristic~$0$, the monoid algebra $\kk \Mon(X,\sigma)$ and the graded space of invariants $\oplus_{k \ge 0} (\kk X^{\times k})^{S_k}$ are linearly isomorphic via the symmetrizers $\frac{1}{\# S_k}\sum_{s \in S_k} s$. The two results are covered by the following easy generalization. Take a braiding~$\sigma$ and linear combinations $P_k \in \kk B^+_k$ such that 
\begin{itemize}
 \item the coefficients of each~$P_k$ sum up to~$1$;
 \item for any $b \in B^+_k$, the actions of $b P_k$, $P_k b$, and $P_k$ on $\kk X^{\times k}$ via~$\sigma$ coincide.
\end{itemize} 
Then the actions of~$P_k$ induce a linear bijection between $\kk \Mon(X,\sigma)$ and $\oplus_{k \ge 0} (\kk X^{\times k})^{B^+_k}$.
\end{remark}

\begin{notation}\label{N:ast}
The associative product on~$\Norm^{+} (X,\sigma)$ corresponding under the bijection~$\Delta_*$ to the concatenation on~$\SG(X,\sigma)$ is denoted by~$\ast$. Explicitly, for $\boldsymbol{v},\boldsymbol{w} \in \Norm^{+} (X,\sigma)$ of length~$n$ and~$m$ respectively, one has $\boldsymbol{v} \ast \boldsymbol{w} = \Delta_{n+m}(\boldsymbol{vw})$. The same notation~$\ast$ is used for the analogous product on $\Norm (X,\sigma)$.
\end{notation}

\begin{definition}\label{D:BrMonoid}
A \emph{braided semigroup} is a semigroup $(M,\cdot)$ endowed with a braiding~$\sigma$, subject to the following compatibility conditions for all $u,v,w \in M$ (Figure~\ref{P:BrSemigroup}):
\begin{align}
\sigma (u \cdot v,w) &= (w'',u' \cdot v'), &&\text{where } \sigma(v,w)=(w',v'), \; \sigma(u,w')=(w'',u');\label{E:BrMonoid}\\
\sigma (u,v \cdot w) &= (v' \cdot w',u''), &&\text{where } \sigma(u,v)=(v',u'), \; \sigma(u',w)=(w',u'').\label{E:BrMonoid'}\\
\intertext{It is called a \emph{braided monoid} if the operation~$\cdot$ admits a unit~$1$ compatible with~$\sigma$: $\sigma(v,1)=(1,v)$, $\sigma(1,v) = (v,1)$. A braided semigroup or monoid is declared \emph{braided commutative} if one more compatibility condition is satisfied:}
w' \cdot v' &= v \cdot w, &&\text{where } \sigma(v,w)=(w',v').\label{E:BrMonoid''}
\end{align}
\end{definition}

\begin{figure}[!h]\centering
\begin{tikzpicture}[yscale=-.4,xscale=.4]
 \draw[rounded corners] (1,0) -- (2,1)-- (2,2.5);
 \draw[rounded corners] (0.5,1.5) -- (1,2) -- (1,2.5);
 \draw[rounded corners] (2,0) -- (0,2)-- (0,2.5);
 \node at (3.7,1.25) {$=$};
\end{tikzpicture}
\begin{tikzpicture}[yscale=-.4,xscale=.4]
 \draw[rounded corners] (0,-0.5) -- (0,0) -- (1.65,1.65) -- (2,2);
 \draw[rounded corners] (2,-0.5) -- (2,1)-- (1,2);
 \draw[rounded corners] (2,0)-- (0,2);
 \node at (6.5,1) {}; 
\end{tikzpicture}
\begin{tikzpicture}[yscale=-.4,xscale=.4]
 \draw[rounded corners] (0,0) -- (2,2)-- (2,2.5);
 \draw[rounded corners] (1.5,1.5)  -- (1,2) -- (1,2.5);
 \draw[rounded corners] (1,0) -- (0,1)-- (0,2.5);
 \node at (3,1.25) {$=$};
\end{tikzpicture}
\begin{tikzpicture}[yscale=-.4,xscale=.4]
 \draw[rounded corners] (0,0) --   (2,2);
 \draw[rounded corners] (0,-0.5) -- (0,1) --  (1,2);
 \draw[rounded corners] (1,-0.5) -- (1,1) -- (0,2);
 \node at (5.5,1) {}; 
\end{tikzpicture}
\begin{tikzpicture}[yscale=-.35,xscale=.35]
 \draw[rounded corners] (0,-1) -- (0,0);
 \draw[rounded corners] (0,0) -- (1,1) -- (-1,2);
 \draw[rounded corners] (0,0) -- (-1,1) -- (1,2);
 \node at (2,.5) {$=$}; 
\end{tikzpicture}
\begin{tikzpicture}[yscale=-.35,xscale=.35]
 \draw[rounded corners] (0,-1) -- (0,0);
 \draw[rounded corners] (0,0) -- (1,1) -- (1,2);
 \draw[rounded corners] (0,0) -- (-1,1) -- (-1,2);
\end{tikzpicture}
\caption{Axioms for a braided commutative semigroup}\label{P:BrSemigroup}
\end{figure}
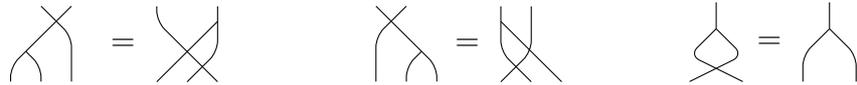

One recovers usual commutative semigroups taking as~$\sigma$ the flip $(u,v) \mapsto (v,u)$.

Everything is now ready for the central result of this section.

\begin{theorem}\label{T:IdempotentBraiding}
\begin{enumerate}
\item The structure semigroup $\SG(X,\sigma)$ of a braided set $(X,\sigma)$ is braided commutative, with the braiding~$\osigmas$ induced by the braiding~$\osigma$ on~$X^*$ (Figure~\ref{P:YBE}\rcircled{B}).
\item If~$\sigma$ is idempotent, then the semigroup of $\sigma$-normal words $(\Norm^{+}(X,\sigma),\ast)$ is braided commutative, with as braiding the restriction~$\osigman$ of~$\osigma$ to $\Norm^{+}(X,\sigma)$. Moreover, the braiding~$\osigman$ followed by concatenation recovers the product~$\ast$.
\item For an idempotent~$\sigma$, the braided semigroups from the previous points are isomorphic, via the map~$\Delta_*$ from Proposition~\ref{PR:DeltaNormal}.
\end{enumerate}
Analogous statements hold for the monoids $\Mon(X,\sigma)$ and $\Norm(X,\sigma)$.
\end{theorem}

\begin{proof}
To show that~$\osigma$ descends to the quotient $\SG(X,\sigma)$ of~$X^*$, one should check its naturality with respect to~$\sigma$ applied at different positions of the arguments. This is done in Figure~\ref{P:Naturality} (where all the strands are labeled by~$X$): the bottom crossing is pulled through a multiple crossing by a sequence of R$\mathrm{III}$ moves---i.e., applications of the YBE. Axioms \eqref{E:BrMonoid}-\eqref{E:BrMonoid''} expressing the compatibility between the induced braiding~$\osigmas$ and the concatenation product are straightforward. Point~1 follows.
\begin{figure}[!h]\centering
\begin{tikzpicture}[xscale=0.4,yscale=0.3]
\draw [rounded corners, shift={(-0.5,0)}] (2,0)--(2,1)--(6,5)--(6,6);
\draw [rounded corners, shift={(-0.5,0)}] (3,0)--(3,1)--(7,5)--(7,6);
\draw [rounded corners, shift={(-0.5,0)}] (4,0)--(4,1)--(8,5)--(8,6);
\draw [rounded corners] (5,0)--(5,1)--(1,5)--(1,6);
\draw [rounded corners, myred, thick] (7,0)--(6,1)--(2,5)--(2,6);
\draw [rounded corners, myred, thick] (6,0)--(7,1)--(3,5)--(3,6);
\draw [rounded corners] (8,0)--(8,1)--(4,5)--(4,6);
\node at (12,3){$\Large =$};
\node at (15,0){ };
\end{tikzpicture} 
\begin{tikzpicture}[xscale=0.4,yscale=0.3]
\draw [rounded corners, shift={(-0.5,0)}] (2,0)--(2,1)--(6,5)--(6,6);
\draw [rounded corners, shift={(-0.5,0)}] (3,0)--(3,1)--(7,5)--(7,6);
\draw [rounded corners, shift={(-0.5,0)}] (4,0)--(4,1)--(8,5)--(8,6);
\draw [rounded corners] (5,0)--(5,1)--(1,5)--(1,6);
\draw [rounded corners, myred, thick] (6,0)--(6,1)--(2,5)--(3,6);
\draw [rounded corners, myred, thick] (7,0)--(7,1)--(3,5)--(2,6);
\draw [rounded corners] (8,0)--(8,1)--(4,5)--(4,6);
\end{tikzpicture} 
\caption{Naturality of~$\osigma$ with respect to a~$\sigma_i$}\label{P:Naturality}
\end{figure}
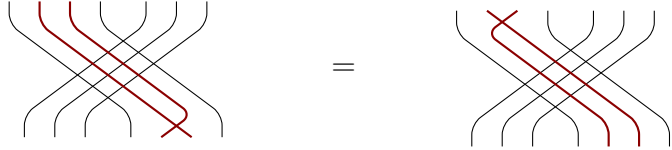

Now assume~$\sigma$ idempotent. An easy graphical argument using Figures~\ref{P:YBE}\rcircled{B} and~\ref{P:Delta} yields the following relations in~$B_{m+n}^+$ (and hence in~$C_{m+n}$):
\begin{align}\label{E:DeltaDecomposition}
&\Delta_{m+n} = (\Delta_n \times \Delta_m) b_{m,n} =b_{m,n} (\Delta_m \times \Delta_n).
\end{align}
This implies that the braiding~$\osigma$ restricts to $\Norm^{+}(X,\sigma)$, and that the bijection $\Delta_* \colon \SG(X,\sigma) \overset{1:1}{\longrightarrow} \Norm^{+}(X,\sigma)$ sends~$\osigmas$ to the restricted braiding~$\osigman$. Further, $\Delta_*$ sends the concatenation to the product~$\ast$ by the definition of the latter. This yields Point~3 and, as a consequence, the first part of Point~2.

It remains to show that, for $\sigma$-normal words $\boldsymbol{v} \in X^{\times m}$ and $\boldsymbol{w} \in X^{\times n}$ with $b_{m,n} (\boldsymbol{v},\boldsymbol{w}) = (\boldsymbol{w'},\boldsymbol{v'})$, the word~$\boldsymbol{w'v'}$ is the normal form of~$\boldsymbol{vw}$, which is $\Delta_{m+n}(\boldsymbol{vw})$ (Proposition~\ref{PR:DeltaNormal}). Relation~\eqref{E:DeltaDecomposition} and the $\sigma$-normality of $\boldsymbol{v}$ and~$\boldsymbol{w}$ yield
$$\Delta_{m+n}(\boldsymbol{vw}) =  b_{m,n} (\Delta_m (\boldsymbol{v}), \Delta_n (\boldsymbol{w})) = b_{m,n} (\boldsymbol{v},\boldsymbol{w}) = \boldsymbol{w'}\boldsymbol{v'}.\qedhere$$ 
\end{proof}

\begin{example}\label{EX:Bubble}
Take a totally ordered set~$X$ and the operator 
$$\sigma(x,y) = (\, \min \{x,y \} \, , \, \max \{x,y \} \,)$$
on~$X^{\times 2}$. It is an idempotent braiding. Indeed, when applied to a triple $(x,y,z)$, both sides of the YBE~\eqref{E:YBE} rearrange its elements in the increasing order. The Coxeter monoids~$C_k$ thus act on the powers~$X^{\times k}$. The structure monoid coincides with the symmetric monoid $S(X) = \langle \, X \, | \, xy = yx \,\rangle$ here, the $\sigma$-normal words are precisely the ordered words, and the map~$\Delta_*$ realizes the bubble sort algorithm. Theorem~\ref{T:IdempotentBraiding} then yields a braiding on~$S(X)$. It also explains how to simplify the bubble sort algorithm when some parts of the sequence to be sorted are already dealt with.
\end{example}

\begin{example}\label{EX:Lattice}
Generalizing the previous example, one can take a distributive lattice $(X,\wedge,\vee)$ and the operator 
$$\sigma(x,y) = (\, x \wedge y \, , \, x \vee y \,)$$
on~$X^{\times 2}$. The distributive lattice axioms force it to be an idempotent braiding. Here $\sigma$-normal words are precisely the ordered words, for the partial order induced by the lattice structure. Sets with the intersection and union operations, and integers with GCDs and LCMs, are two important examples. For integers, the map~$\Delta_*$ yields a recipe for computing the Smith normal form of a diagonal matrix over~$\ZZ$, provided that one knows how to do it for $2 \times 2$ matrices.
\end{example}

In the next example, as well as in the case of Young tableaux~\cite{LebedPlactic}, braided sets contain a ``dummy'' element (e.g., the unit or the empty row), which one needs to be able to get rid of. We now develop techniques for doing so. They are to be compared with the \emph{epinormalization} of \cite{DehGui}.

\begin{definition}
An idempotent braided set $(X,\sigma)$ is called \emph{pseudo-unital}, or \emph{PUIBS}, if it is endowed with a \emph{pseudo-unit}, i.e., an element $1 \in X$ satisfying:
\begin{enumerate}
\item both $\sigma(1,x)$ and $\sigma(x,1)$ lie in $\{\, (1,x),(x,1) \,\}$ for all $x\in X$;
\item a normal word with any occurrence of the letter~$1$ omitted remains normal.
\end{enumerate}
Given a word $\boldsymbol{w} \in X^*$, let the word $\boldsymbol{\ow}$ be obtained from it by erasing all its letters $1$. Denote by $\oNorm(X,\sigma,1)$ the set of normal words avoiding the letter $1$. Finally, let $\oMon(X,\sigma,1)$ be the monoid $\Mon(X,\sigma)$ with the letter $1$ identified with the empty word. It will be called the \emph{reduced structure monoid} of $(X,\sigma,1)$.
\end{definition}

Condition~1 implies in particular $\sigma(1,1) = (1,1)$.

Condition~2 means that the map $\boldsymbol{w} \mapsto \boldsymbol{\ow}$ yields a surjection $\Norm(X,\sigma) \twoheadrightarrow \oNorm(X,\sigma,1)$, with a tautological section. We will next show that, at the level of structure monoid, this corresponds to the quotient $\Mon(X,\sigma) \twoheadrightarrow \oMon(X,\sigma,1)$.

\begin{proposition}\label{PR:PUIBS}
Let $(X,\sigma,1)$ be a PUIBS. Then the following commutative diagram can be completed in a unique way:
$$\xymatrix @!0 @R=1.2cm @C=2cm{
    \Mon(X,\sigma) \ar[rr]^-{\Delta_*}_-{1:1} \ar@{->>}[d] && \Norm(X,\sigma) \ar@{->>}[d] & \boldsymbol{w} \ar@{|->}[d] \\
    \oMon(X,\sigma,1) \ar@{.>}[rr]^-{\exists ! \, \oDelta} && \oNorm(X,\sigma,1) & \boldsymbol{\ow}
}$$
Moreover, the induced map $\oDelta$ is necessarily a bijection.
\end{proposition}

\begin{proof}
First we show that the bijection~$\Delta_*$ followed by the map $\boldsymbol{w} \mapsto \boldsymbol{\ow}$ induces a map $\oDelta \colon \oMon(X,\sigma,1) \to \oNorm(X,\sigma,1)$. For this, let us check that for a word $\boldsymbol{w'}$ obtained from $\boldsymbol{w}$ by inserting a letter~$1$ at position~$p$, its normal form $\Delta_{k+1}(\boldsymbol{w'})$ differs from $\Delta_k(\boldsymbol{w})$ by one letter~$1$ as well. Write $\Delta_{k+1}(\boldsymbol{w'})$ as $b_{i_t}\cdots b_{i_1}(\boldsymbol{w'})$. Recall that the generators $b_i$ of the Coxeter monoid~$C_{k+1}$ act on $X^{\times (k+1)}$ as the braiding~$\sigma$ applied to the components $i$ and $i+1$. Put $\boldsymbol{w'}_{(j)} = b_{i_j}\cdots b_{i_1}(\boldsymbol{w'})$, $\boldsymbol{w'}_{(0)} = \boldsymbol{w'}$. The position sequence $u(j)$ will describe how the inserted letter~$1$ moves in the word sequence $\boldsymbol{w'}_{(j)}$. Concretely, define $u$ inductively by $u(0)=p$, and $u(j)= u(j-1)$ except when 
\begin{itemize}
\item $i_j = u(j-1) - 1$ and $\boldsymbol{w'}_{(j)} \neq \boldsymbol{w'}_{(j-1)}$, in which case put $u(j)=u(j-1) - 1$;
\item $i_j = u(j-1)$ and $\boldsymbol{w'}_{(j)} \neq \boldsymbol{w'}_{(j-1)}$, in which case put $u(j)=u(j-1) + 1$.
\end{itemize}
Finally, let $\boldsymbol{w}_{(j)}$ be obtained from $\boldsymbol{w'}_{(j)}$ by deleting the letter~$1$ at position $u(j)$. The definition of pseudo-unit implies that $\boldsymbol{w}_{(j)}$ differs from $\boldsymbol{w}_{(j-1)}$ by at most one application of a generator of~$C_k$, and that $\boldsymbol{w}_{(t)}$ is normal since $\boldsymbol{w'}_{(t)} = \Delta_{k+1}(\boldsymbol{w'})$ is so. Hence $\boldsymbol{w}_{(t)}$ is the normal form of $\boldsymbol{w}_{(0)} = \boldsymbol{w}$, and we are done.

The uniqueness of~$\oDelta$ is obvious. Its surjectivity follows from that of the maps $\Delta_* \colon \Mon(X,\sigma) \overset{1:1}{\longrightarrow} \Norm(X,\sigma)$ and $\Norm(X,\sigma) \twoheadrightarrow \oNorm(X,\sigma,1)$. As for injectivity, observe that a word $\boldsymbol{w} \in X^{\times k}$ and its reduced normalization $\overline{\Delta_k(\boldsymbol{w})}$ represent the same element of $\oMon(X,\sigma,1)$.
\end{proof}

\begin{notation}\label{N:PUIBSast}
Let~$\ast$ be the associative product on $\oNorm(X,\sigma,1)$ corresponding to the concatenation on $\oMon(X,\sigma,1)$ via the induced bijection~$\oDelta$. Its unit is the empty word, denoted by~$\emptyw$.
\end{notation}

\begin{example}\label{EX:Factorization}
Consider a monoid factorization $G=HK$. That is, $H$ and~$K$ are submonoids of~$G$, and any $g\in G$ uniquely decomposes as $g=hk$ with $h \in H, k \in K$. Put $X = H \cup K$, and, for $x,y \in X$, set $$\sigma(x,y) = (y',x'), \text{ where } y' \in H, x' \in K,\, y'x'=xy.$$ 
This is an idempotent braiding: applied to a triple $(x,y,z)$, both sides of the YBE yield $(z',1,x')$, where $z' \in H$ and $x' \in K$ form the unique $HK$-decomposition of~$xyz$. The normal form of a word $x_{1}\ldots x_{p}$, $p \ge 2$, is $h 1\ldots 1k$, where $hk$ is the $HK$-decomposition of the total product $x_{1}\cdots x_{p}$, and $p-2$ letters~$1$ are inserted in the middle. This explicit form makes it obvious that the unit $1$ of~$G$ is a pseudo-unit for $(X,\sigma)$, and that the reduced monoid $\oMon(X,\sigma,1)$ recovers~$G$. In this particular case the braiding~$\osigmas$ survives in this quotient, turning $G$ into a braided commutative semigroup. It is not a braided monoid: one has $\sigma(x,1) = (x,1) \neq (1,x)$ for $x \in H \setminus \{1\}$. The map~$\oDelta$ yields here a factorizing procedure for a multi-term product. For the trivial factorization $G=\{1\} \, G$, one recovers the braiding $\sigma(g,g') = (1,gg')$ on~$G$, which encodes the associativity, as explained in~\cite{Lebed1}. This braiding yields a $C_k$-action on~$G^{\times k}$.
\end{example} 

\begin{example}\label{EX:Size2}
To show the diversity of idempotent braidings, we give their complete classification on a two-element set $X=\{0,1\}$. Up to isomorphism, they are $16$. Each braiding is written in a way which suggests how to generalize it to larger sets.
\begin{enumerate}
\item $\sigma = \Id_{X \times X}$;
\item $\sigma (x,y) = (0,0)$;
\item $\sigma(x,y) = (x,f(x))$, \qquad $\sigma(x,y) = (f(y),y)$, 

 where $F \colon X \to X$ is one of the $3$ maps $x \mapsto x$, $x \mapsto x+1$, $x \mapsto 0$;
\item $\sigma(x,y) = (x \diamond y,0)$, \qquad $\sigma(x,y) = (0, x \diamond y)$,

 where~$\diamond$ is one of the $2$ operations $+$ and $\max$;
\item $\sigma(x,y) = (\, \min \{x,y \} \, , \, y \,)$, \qquad $\sigma(x,y) = (\, x \, , \, \min \{x,y \} \,)$;
\item $\sigma(x,y) = (\, \max \{x,y \} \, , \, \max \{x,y \} \,)$;
\item $\sigma(x,y) = (\, \min \{x,y \} \, , \, \max \{x,y \} \,)$.
\end{enumerate}
\end{example} 


\section{Basics of braided (co)homology}\label{S:BrHom}

We now recall the braided (co)homology constructions from~\cite{HomologyYB,Lebed1}. The following objects play the role of coefficients in these theories.

\begin{definition}
A \emph{right module} over a braided set $(X, \sigma)$ is a pair $(M, \rho)$, where $M$ is a set and $\rho \colon M \times X \to M$, $(m,x) \mapsto m \cdot x$, is a map compatible with~$\sigma$:
$$(m \cdot x) \cdot y = (m \cdot  y') \cdot x' \qquad\qquad \text{ for all } m \in M,\, x,y \in X, \, (y',x') = \sigma(x,y).$$ 
\emph{Left modules} $(M, \,\lambda \colon X \times M \to M)$ are defined similarly. A \emph{bimodule} $(M,\rho,\lambda)$ over $(X, \sigma)$ combines commuting right and left module structures, in the sense of $(x \cdot m) \cdot y = x \cdot (m \cdot y)$ (Figure~\ref{P:BrMod}). A (bi)module~$M$ is called \emph{linear} if it is an abelian group and the maps $m \mapsto m \cdot x$ and/or $m \mapsto x \cdot m$ are linear for all $x \in X$. Such a bimodule is called a \emph{bimodule-algebra} if it is endowed with a bilinear associative product~$\mu$ 
 satisfying
\begin{align*}
\mu(x \cdot m_1,m_2) &= x \cdot \mu(m_1,m_2), \\
\mu(m_1,m_2 \cdot x) &= \mu(m_1,m_2) \cdot x, \\
\mu(m_1 \cdot x,m_2) &= \mu(m_1,x \cdot m_2).
\end{align*}
(Bi)modules over a PUIBS $(X, \sigma, 1)$ are required to satisfy $ m \cdot 1 = 1 \cdot m = m$.
\end{definition}

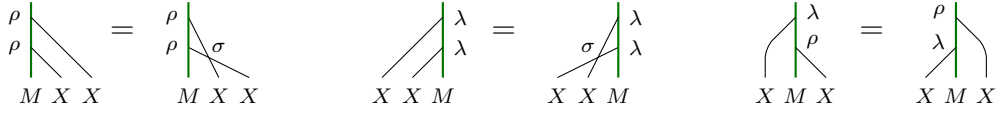
\begin{figure}[!h]\centering
\begin{tikzpicture}[xscale=0.4,yscale=0.4]
 \draw [thick, mygreen] (0,0) -- (0,2.5);
 \draw (1,0) -- (0,1);
 \draw (2,0) -- (0,2);
 \node at (0,2) [left] {$\scriptstyle \rho$};
 \node at (0,1) [left] {$\scriptstyle \rho$};
 \node at (2,0) [below] {$\scriptstyle X$};
 \node at (1,0) [below] {$\scriptstyle X$};
 \node at (0,0) [below] {$\scriptstyle M$};
 \node  at (3,1.5){$=$};
\end{tikzpicture}
\begin{tikzpicture}[xscale=0.4,yscale=0.4]
 \node  at (-0.5,1.5){};
 \draw (1,0) -- (0,2);
\draw (2,0) -- (0,1);
 \draw [thick, mygreen] (0,0) -- (0,2.5);
 \node at (0,1) [left] {$\scriptstyle \rho$};
 \node at (0,2) [left] {$\scriptstyle \rho$};
 \node at (2,0) [below] {$\scriptstyle X$};
 \node at (1,0) [below] {$\scriptstyle X$};
 \node at (0,0) [below] {$\scriptstyle M$};
 \node at (1,1) {$\scriptstyle \sigma$}; 
\end{tikzpicture}
\begin{tikzpicture}[xscale=0.4,yscale=0.4]
 \node  at (-2,1.5){};
 \draw [thick, mygreen] (3,0) -- (3,2.5);
 \draw (1,0) -- (3,2);
 \draw (2,0) -- (3,1);
 \node at (3,2) [right] {$\scriptstyle \lambda$};
 \node at (3,1) [right] {$\scriptstyle \lambda$};
 \node at (2,0) [below] {$\scriptstyle X$};
 \node at (1,0) [below] {$\scriptstyle X$};
 \node at (3,0) [below] {$\scriptstyle M$};
 \node  at (5,1.5){$=$};
\end{tikzpicture}
\begin{tikzpicture}[xscale=0.4,yscale=0.4]
 \draw (1,0) -- (3,1); 
 \draw [thick, mygreen] (3,0) -- (3,2.5);
 \draw (2,0) -- (3,2);
 \node at (3,2) [right] {$\scriptstyle \lambda$};
 \node at (3,1) [right] {$\scriptstyle \lambda$};
 \node at (2,0) [below] {$\scriptstyle X$};
 \node at (1,0) [below] {$\scriptstyle X$};
 \node at (3,0) [below] {$\scriptstyle M$};
 \node at (2,1) {$\scriptstyle \sigma$}; 
\end{tikzpicture}
\begin{tikzpicture}[xscale=0.4,yscale=0.4]
 \node  at (-1,1.5){};
 \draw [thick, mygreen] (3,0) -- (3,2.5);
 \draw [rounded corners] (2,0) -- (2,1) -- (3,2);
 \draw (4,0) -- (3,1);
 \node at (3,2.2) [right] {$\scriptstyle \lambda$};
 \node at (3,1.2) [right] {$\scriptstyle \rho$};
 \node at (2,0) [below] {$\scriptstyle X$};
 \node at (4,0) [below] {$\scriptstyle X$};
 \node at (3,0) [below] {$\scriptstyle M$};
 \node  at (5.5,1.5){$=$};
\end{tikzpicture}
\begin{tikzpicture}[xscale=0.4,yscale=0.4]
 \draw [thick, mygreen] (3,0) -- (3,2.5);
 \draw [rounded corners] (4,0) -- (4,1) -- (3,2);
 \draw (2,0) -- (3,1);
 \node at (3,1.2) [left] {$\scriptstyle \lambda$};
 \node at (3,2.2) [left] {$\scriptstyle \rho$};
 \node at (2,0) [below] {$\scriptstyle X$};
 \node at (4,0) [below] {$\scriptstyle X$};
 \node at (3,0) [below] {$\scriptstyle M$};
\end{tikzpicture}
\caption{A bimodule $(M, \rho, \lambda)$ over $(X, \sigma)$.}\label{P:BrMod}
\end{figure}

\begin{remark}
(Bi)modules over $(X, \sigma)$ can be regarded as (bi)modules over the structure semigroup $\SG(X,\sigma)$ or the structure monoid $\Mon(X,\sigma)$. Similarly, (bi)modules over a PUIBS $(X, \sigma, 1)$ correspond to (bi)modules over $\oSG(X,\sigma,1)$.
\end{remark}

\begin{example}\label{EX:Trivial}
Any set~$M$ equipped with the projections $M \times X \to M$, $m \cdot x = m$, and $X \times M \to M$, $x \cdot m = m$, is an $(X,\sigma)$-bimodule. This structure is called \emph{trivial}.
\end{example}

\begin{example}\label{EX:Adjoint}
The braided set $(X,\sigma)$ is a right and a left module over itself, with the actions $\rho \colon (x,y) \mapsto x'$ and $\lambda \colon (x,y) \mapsto y'$, where $(y',x') = \sigma(x,y)$. These actions do not always combine into a bimodule structure. More generally, the powers $X^{\times k}$ are right and left module over~$X^*$, with the module structure adjoint to the extension~$\osigma$ of the braiding~$\sigma$ to~$X^*$ (Figure~\ref{P:YBE}\rcircled{B}). The linearized sets $\ZZ X^{\times k}$ receive induced linear $(X^*,\osigma)$-module structures. All these modules are baptized \emph{adjoint}. 
\end{example}

\begin{example}\label{EX:StructureSGAsBimodAlg}
The structure semigroup $\SG(X,\sigma)$ is an $(X,\sigma)$-bimodule, with the concatenation actions. Its linearization $\ZZ \SG(X,\sigma)$ becomes a bimodule-algebra. The same is true about the structure monoid $\Mon(X,\sigma)$.
\end{example}

As usual, we let the positive braid monoid~$B^+_k$ from~\eqref{E:Bn} act on $X^{\times k}$ via~$\sigma$.

\begin{theorem}[\cite{HomologyYB,Lebed1}]\label{T:BrHom}
Take a braided set $(X,\sigma)$.
\begin{enumerate}
\item Let $(M,\rho)$ and $(N,\lambda)$ be linear right and left modules over $(X,\sigma)$ respectively. Consider the abelian groups $C_k = M \otimes_{\ZZ} \ZZ X^{\times k} \otimes_{\ZZ} N$, $k \ge 0$, and the linear maps 
\begin{align*}
d_k &= \sum\nolimits_{i=1}^{k}(-1)^{i-1}(d_{k;i}^{l} - d_{k;i}^{r}) \colon C_k \to C_{k-1}, \quad k > 0,\\
\text{where }\quad d_{k;i}^{l} &(m,x_1,\ldots,x_k,n) = (m \cdot x'_i,x'_1,\ldots,x'_{i-1},x_{i+1},\ldots,x_k,n),\\
&\qquad x'_i x'_1 \ldots x'_{i-1} = b_1 \cdots b_{i-1}(x_1 \ldots x_i),\\
 d_{k;i}^{r}&(m,x_1,\ldots,x_k,n) = (m,x_1,\ldots,x_{i-1},x''_{i+1},\ldots,x''_k,x''_i \cdot n),\\
&\qquad x''_{i+1} \ldots x''_k x''_i = b_{k-i} \cdots b_1(x_i \ldots x_k)
\end{align*}
(Figure~\ref{P:BrHom}), completed by $d_0 = 0$. They form a chain complex.
\item If $(M,\rho,\lambda)$ is a linear $(X,\sigma)$-bimodule, then a similar differential can be defined on $C_k = M \otimes_{\ZZ} \ZZ X^{\times k}$: one simply replaces $d_{k;i}^{l}$ and~$d_{k;i}^{r}$ with
\begin{align*}
d_{k;i}^{l} &(m,x_1,\ldots,x_k) = (m \cdot x'_i,x'_1,\ldots,x'_{i-1},x_{i+1},\ldots,x_k),\\
d_{k;i}^{r} &(m,x_1,\ldots,x_k) = (x''_i \cdot m,x_1,\ldots,x_{i-1},x''_{i+1},\ldots,x''_k).
\end{align*}
\item Let $(M,\rho,\lambda)$ be a linear $(X,\sigma)$-bimodule. Consider the abelian groups $C^k = \Map (X^{\times k},M)$, $k \ge 0$, with $\Map (X^{\times 0},M)$ interpreted as~$M$, and the linear maps 
\begin{align*}
d^{k} &= \sum\nolimits_{i=1}^{k}(-1)^{i-1}(d^{k;i}_{l} - d^{k;i}_{r}) \colon C^{k-1} \to C^{k}, \quad k > 0,\\
\text{where }\quad (d^{k;i}_{l}f) &(x_1,\ldots,x_k) = x'_i \cdot f(x'_1,\ldots,x'_{i-1},x_{i+1},\ldots,x_k),\\
(d^{k;i}_{r}f) &(x_1,\ldots,x_k) = f(x_1,\ldots,x_{i-1},x''_{i+1},\ldots,x''_k) \cdot x''_i,
\end{align*}
with notations from Point~1. They form a cochain complex.
\end{enumerate}
\end{theorem}

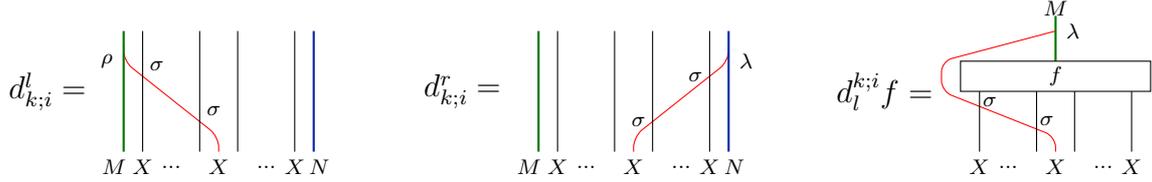
\begin{figure}[!h]\centering
\begin{tikzpicture}[xscale=0.25,yscale=0.2]
 \draw [red,rounded corners] (-1,0) -- (-1,-1) -- (4,-6) -- (4,-7);
 \draw [thick, mygreen] (-1,1) --  (-1,-7);
 \draw [thick, myblue] (9,1) --  (9,-7); 
 \draw (0,1) --  (0,-7);
 \draw (3,1) -- (3,-7);
 \draw (5,1) -- (5,-7);
 \draw (8,1) -- (8,-7);
 \node at (-5,-3) {$d_{k;i}^{l}=$};
 \node at (-1,-1)  [left] {$\scriptstyle \rho$};
 \node at (-0.2,-2.3) [above right] {$\scriptstyle \sigma$};
 \node at (2.8,-5.3) [above right] {$\scriptstyle \sigma$};
 \node at (0,-8) {$\scriptstyle X$};
 \node at (1.5,-8) {$\scriptstyle \ldots$};
 \node at (4,-8) {$\scriptstyle X$};
 \node at (6.5,-8) {$\scriptstyle\ldots$};
 \node at (8,-8) {$\scriptstyle X$};
 \node at (-1.5,-8) {$\scriptstyle M$};
 \node at (9.3,-8) {$\scriptstyle N$};   
\end{tikzpicture} \qquad
\begin{tikzpicture}[xscale=0.25,yscale=0.2]
 \draw [red,rounded corners] (9,0) -- (9,-1) -- (4,-6) -- (4,-7);
 \draw [thick, mygreen] (-1,1) --  (-1,-7);
 \draw [thick, myblue] (9,1) --  (9,-7);
 \draw (0,1) -- (0,-7);
 \draw (3,1) -- (3,-7);
 \draw (5,1) -- (5,-7);
 \draw (8,1) -- (8,-7);
 \node at (-5,-3) {$d_{k;i}^{r}=$};
 \node at (9,-1) [right] {$\scriptstyle \lambda$};
 \node at (5.2,-5) [left] {$\scriptstyle \sigma$};
 \node at (8.2,-2) [left] {$\scriptstyle \sigma$};
 \node at (-1.5,-8) {$\scriptstyle M$};
 \node at (0,-8) {$\scriptstyle X$};
 \node at (1.5,-8) {$\scriptstyle \ldots$};
 \node at (4,-8) {$\scriptstyle X$};
 \node at (6.5,-8) {$\scriptstyle\ldots$};
 \node at (8,-8) {$\scriptstyle X$}; 
 \node at (9.3,-8) {$\scriptstyle N$};
\end{tikzpicture}\qquad
\begin{tikzpicture}[xscale=0.25,yscale=0.2]
 \draw [thick, mygreen] (4,-1) --  (4,2);
 \draw (-1,-3) rectangle (9,-1);
 \node at (4,-2) {$\scriptstyle f$}; 
 \draw [red,rounded corners] (4,1) -- (-2,-1) -- (-2,-3) -- (4,-6) -- (4,-7);
 \draw (0,-3) --  (0,-7);
 \draw (3,-3) -- (3,-7);
 \draw (5,-3) -- (5,-7);
 \draw (8,-3) -- (8,-7);
 \node at (-5,-3) {$d^{k;i}_{l} f =$};
 \node at (4,1)  [right] {$\scriptstyle \lambda$};
 \node at (-0.4,-4.7) [above right] {$\scriptstyle \sigma$};
 \node at (2.6,-6) [above right] {$\scriptstyle \sigma$};
 \node at (0,-8) {$\scriptstyle X$};
 \node at (1.5,-8) {$\scriptstyle \ldots$};
 \node at (4,-8) {$\scriptstyle X$};
 \node at (6.5,-8) {$\scriptstyle\ldots$};
 \node at (8,-8) {$\scriptstyle X$};
 \node at (4,2.5) {$\scriptstyle M$}; 
\end{tikzpicture}
   \caption{The $i$th terms of braided differentials. The $i$th $X$-strand moves to the left / to the right of all other $X$-strands, and then acts on coefficients.}\label{P:BrHom}
\end{figure}

\begin{definition}
The (co)homology groups of the complexes above, denoted by $H_k(X,\sigma;$ $M,N)$, %
 $H_k(X,\sigma;M)$, and $H^k(X,\sigma;M)$, are referred to as the \emph{braided (co)homology groups} of $(X,\sigma)$ with coefficients in $(M,N)$ or in~$M$ respectively.
\end{definition}

The theorem can be proved by easy diagram manipulations; an alternative argument will be sketched in the next section.

To give a better feeling of the braided differentials, we propose explicit formulas for a $g \in \Map (X^{\times 0},M)$ represented by $m \in M$, and for an $f \in \Map (X,M)$:
\begin{align}
&d^1g(x) = x \cdot m - m \cdot x, \notag\\
&d^2f(x) = x_1 \cdot f(x_2) + f(x_1) \cdot x_2 - x_2' \cdot f(x_1') - f(x_2') \cdot x_1',\notag\\
\intertext{where $(x_2',x_1') = \sigma (x_1,x_2)$. If the bimodule~$M$ is trivial (in the sense of Example~\ref{EX:Trivial}), then the formulas simplify: $d^1g(x) = 0$,}
&d^2f(x) = f(x_1) + f(x_2) -f(x_2') - f(x_1'). \label{E:d2}
\end{align}

It will sometimes be convenient to separate the $l$- and the $r$-terms (referred to as \emph{left} and \emph{right}, for diagrammatic reasons) in the braided differentials, writing
\begin{align}\label{E:BrDiffCosh}
d_k &= d_k^l + (-1)^k d_k^r, & d^k &= d^k_l + (-1)^k d^k_r.
\end{align}
In these left and right differentials, each term comes with the sign $(-1)^{cr(D)}$, where $cr(D)$ is the crossing number of the corresponding diagram~$D$. This is a Koszul sign: it switches each time two $X$-terms change places (which is realized by an application of the braiding~$\sigma$, i.e., by a crossing in our diagram).

\begin{remark}\label{R:Cub}
As the form of our differentials suggests, they come from a \emph{pre-cubical} structure \cite{Lebed1}. Since it is not essential for this paper, we develop this pre-cubical viewpoint as a series of remarks only.
\end{remark}

\begin{example}\label{EX:FactorizationDiff}
Let us identify braided differentials for a monoid~$X$ equipped with the braiding $\sigma(x,x') = (1,xx')$ (Example~\ref{EX:Factorization}). As an $(X,\sigma)$-bimodule, take a linear bimodule~$M$ over the monoid $X$. The differentials on $M \otimes_{\ZZ} \ZZ X^{\times k}$ read
\begin{align*}
d_k^l (m,x_1,\ldots,x_k) &= (m \cdot x_1,x_2,\ldots,x_k) - (m,x_1x_2,x_3,\ldots,x_k) + \cdots \\
&\qquad\qquad + (-1)^{k-1} (m,x_1,\ldots,x_{k-2},x_{k-1}x_k), \\
d_k^r (m,x_1,\ldots,x_k) &= (x_k \cdot m,x_1,\ldots,x_{k-1}) + (\text{ some terms containing } 1\text{'s}).
\end{align*}
In~$d_k^l$ one readily recognizes the bar construction. The full differential $d_k^l + (-1)^k d_k^r$ yields the Hochschild complex, after modding out the terms with $x_i=1$ for at least one~$i$. These terms form a subcomplex, which can be checked either directly, or via the critical subcomplex approach (Section~\ref{S:BrHomSG}), or using the degeneracies 
$$s_i \colon (m,x_1,\ldots,x_k) \mapsto (m,x_1,\ldots,x_{i-1},1,x_i,\ldots,x_k)$$ 
from~\cite{Lebed1}. Dually, one recovers the Hochschild cohomology when restricting to the critical subcomplex $CrC^k$ of the maps $X^{\times k} \to M$ vanishing whenever one of the arguments is~$1$.
\end{example}

See~\cite{Lebed1,Lebed2,LebedVendramin} for other examples of (co)homology theories interpreted in the braided framework.

\section{Cup product}\label{S:BrHomCup}

In key cases, braided complexes carry more structure than the bare differential. They thus capture more information about the braided set. This additional structure is best presented using an alternative interpretation of the braided (co)homology, in terms of the {quantum shuffles} of Rosso \cite{Rosso1Short,Rosso2}. 

Concretely, the \emph{shuffle sets} are the permutation sets
$$Sh_{p_1,p_2,\ldots, p_t}=\Bigg\{ s \in S_{p_1+p_2+\cdots+p_t}  \;
\begin{array}{|c}
\scriptstyle s(1)<s(2)<\ldots<s(p_1), \\
\scriptstyle s(p_1+1)<\ldots<s(p_1+p_2), \\
\scriptstyle s(p+1)<\ldots<s(p+p_t)
 \end{array}
 \Bigg\} $$
with $p=p_1+\cdots+p_{t-1}$. Morally, one permutes $p_1+p_2+\cdots+p_{t}$ elements preserving the order within $t$ consecutive blocks of size $p_1, \ldots, p_t$, just like when shuffling cards. Recall further the projection $B_k^+ \twoheadrightarrow S_k$, $b_i \mapsto s_i$, and its set-theoretical section 
\begin{align*}
 S_k & \hookrightarrow  B_k^+,\\
 s = s_{i_1}s_{i_2}\cdots s_{i_t}& \mapsto b_{i_1}b_{i_2}\cdots b_{i_t} =:T_s, 
\end{align*}
where $s_{i_1}s_{i_2}\cdots s_{i_t}$ is any of the shortest words representing $s\in S_k$. Now the \emph{quantum shuffle product} on $\ZZ\langle X \rangle= \oplus_{k \geqslant 0} \ZZ X^{\times k}$
is the $\ZZ$-linear extension of the maps
\begin{align}
\sh\!{\,}_{p,q} &:= \sum_{s\in Sh_{p,q}} T_s \, \colon  X^{\times p} \times X^{\times q} \to \ZZ X^{\times(p+q)}.\label{E:qu_sh}
\end{align}
Explicitly, for $\boldsymbol{w} \in X^{\times p}$ and $\boldsymbol{v} \in X^{\times q}$, we put $\boldsymbol{w}\sh\boldsymbol{v}= \sum_{s\in Sh_{p,q}} T_s (\boldsymbol{w}\boldsymbol{v})$. Dually, the \emph{quantum shuffle coproduct} on $\ZZ\langle X \rangle$ is defined by $\csh|_{X^{\times k}} := \displaystyle \sum_{p+q=k;\:p,q\ge 0} \csh^{p,q}$,
\begin{align}\label{E:qu_cosh}
\csh^{p,q} &:= \sum_{s\in Sh_{p,q}} T_{s^{-1}} \, \colon X^{\times (p+q)} \to \ZZ X^{\times p}\times X^{\times q}.
\end{align} 
Typical terms of $\sh\!{\,}_{p,q}$ and $\csh^{p,q}$ are depicted in Figure~\ref{P:Shuffle}. Replacing the operators~$T_{s}$ with $(-1)^{|s|}T_{s}$ in all the formulas above, one gets a product and a coproduct denoted by~$\Sh$ and~$\Csh$ respectively. For instance, the case $p=2, q=1$ yields 
\begin{align*}
xy \sh z &= xyz + xz'y' + z''x'y', & xy \Sh z &= xyz - xz'y' + z''x'y',
\end{align*}
where $x,y,z \in X$, $\sigma(y,z)=(z',y')$, $\sigma(x,z')=(z'',x')$.

\begin{figure}[!h]\centering
\begin{tikzpicture}[xscale=0.35,yscale=-0.4] 
 \draw [myred] (-1,0) -- (4,3);
 \draw [myblue] (0,0) -- (-1,3);
 \draw [myred] (1,0) -- (5,3);
 \draw [myblue] (2,0) -- (0,3); 
 \draw [myblue] (3,0) -- (1,3); 
 \draw [myred] (4,0) -- (6,3); 
 \draw [myred] (5,0) -- (7,3); 
 \draw [myblue] (6,0) -- (2,3); 
 \draw [myred] (7,0) -- (8,3);
 \node at (3,-.7) {$\scriptstyle X^{\times (p+q)}$};  
 \node at (0.5,3.7) {$\scriptstyle X^{\times p}$};  
 \node at (6,3.7) {$\scriptstyle X^{\times q}$};  
\end{tikzpicture} \hspace*{2cm}
\begin{tikzpicture}[xscale=0.35,yscale=0.4] 
 \draw [myred] (-1,0) -- (4,3);
 \draw [myblue] (0,0) -- (-1,3);
 \draw [myred] (1,0) -- (5,3);
 \draw [myblue] (2,0) -- (0,3); 
 \draw [myblue] (3,0) -- (1,3); 
 \draw [myred] (4,0) -- (6,3); 
 \draw [myred] (5,0) -- (7,3); 
 \draw [myblue] (6,0) -- (2,3); 
 \draw [myred] (7,0) -- (8,3);
 \node at (3,-.7) {$\scriptstyle X^{\times (p+q)}$};  
 \node at (0.5,3.7) {$\scriptstyle X^{\times p}$};  
 \node at (6,3.7) {$\scriptstyle X^{\times q}$};  
\end{tikzpicture} 
   \caption{Quantum shuffle product and coproduct.}\label{P:Shuffle}
\end{figure}
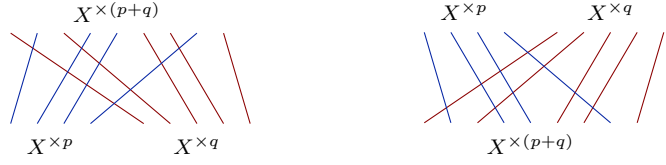

\begin{proposition}[\cite{Rosso1Short}]\label{PR:Shuffle}
The maps~$\sh$ and~$\Sh$ yield an associative product on~$\ZZ\langle X \rangle$. The maps~$\csh$ and~$\Csh$ yield a coassociative coproduct on~$\ZZ\langle X \rangle$.
\end{proposition}

The differentials from Theorem~\ref{T:BrHom}, Point~1 decompose as 
\begin{align*}
d_k &= d_k^l + (-1)^k d_k^r,& d_k^l &= (\rho \times \Id_{\cdots}) \Csh^{1,k-1},& d_k^r &= (\Id_{\cdots} \times \lambda) \Csh^{k-1,1},
\end{align*}
and similarly for Points 2-3. Here and afterwards we use abusive notations of type $\Csh^{p,q} = \Id_M \times \Csh^{p,q} \times \Id_N$. The signed coproduct~$\Csh$ takes care of the Koszul signs in the definition of~$d_k$. The relation $d_{k-1} d_k = 0$ follows from the coassociativity of~$\Csh$ and the relations translating the definition of braided modules:
\begin{align*}
\rho (\rho \times \Id_X) (\Id_M \times \Csh^{1,1}) &= 0, & \lambda (\Id_X \times \lambda) (\Csh^{1,1} \times \Id_N) &= 0.
\end{align*}

We now show how to combine the quantum shuffle coproduct on~$\ZZ\langle X \rangle$ and a product on~$M$ into a cup product on the braided cohomology $H^*(X,\sigma;M)$.

\begin{theorem}\label{T:CupProd}
Let $(M,\rho,\lambda, \mu)$ be a bimodule-algebra over a braided set $(X,\sigma)$. For two maps $f \colon X^{\times p} \to M$ and $g \colon X^{\times q} \to M$, put
\begin{align}
f \smile g& = \mu (f \times g) \Csh^{p,q} \, \colon X^{\times (p+q)} \to M.\label{E:Cup}
\end{align}
This turns $(\bigoplus_{k \ge 0} \Map (X^{\times k},M),d^{k})$ into a differential graded associative algebra, and induces an associative product on $H^* (X,\sigma;M) := \bigoplus_{k \ge 0} H^k(X,\sigma;M)$, also denoted by~$\smile$. For a commutative ring~$\kk$ with trivial $(X,\sigma)$-actions, the product~$\smile$ on $H^* (X,\sigma;\kk)$ is graded commutative, in the sense of 
 \begin{align*}
[f] \smile [g] &= (-1)^{pq} [g] \smile [f] & \text{ for all }\; f \colon X^{\times p} \to \kk, \, g \colon X^{\times q} \to \kk. 
 \end{align*}
\end{theorem}

\begin{definition}
The above products $\smile$ are called \emph{cup products}.
\end{definition}

For trivial coefficients $M = \kk$, the cup product was defined from a completely different viewpoint by Farinati and Garc{\'{\i}}a-Galofre \cite{FarinatiGalofre}. Its commutativity in cohomology was not established in their work.

\begin{proof}
The associativity of~$\smile$ follows from that of~$\mu$ combined with the coassociativity of~$\Csh$. Indeed, extend the definition~\eqref{E:qu_cosh} to
\begin{align*}
\Csh^{p_1,\ldots, p_t} := \sum_{s\in Sh_{p_1,\ldots, p_t}} (-1)^{|s|} T_{s^{-1}} \, \colon X^{\times (p_1+\cdots+p_t)} \to \ZZ X^{\times p_1}\times \cdots \times X^{\times p_t}.
\end{align*} 
Then both sides of the associativity relation for $f \colon X^{\times p} \to M$, $g \colon X^{\times q} \to M$, and $h \colon X^{\times r} \to M$ equal $\mu^2 (f \times g \times h) \Csh^{p,q,r}$, where $\mu^2 = \mu (\mu \times \Id_M)$.

Let us now check that~$\smile$ is compatible with the differentials, in the sense of
\begin{align}\label{E:dgaTotal}
d^{p+q+1}(f \smile g) = d^{p+1}(f) \smile g + (-1)^p f \smile d^{q+1}(g).
\end{align}
Using the decomposition~\eqref{E:BrDiffCosh}, it suffices to establish the relations
\begin{align}
d^{p+q+1}_l(f \smile g) &= d^{p+1}_l(f) \smile g, \label{E:dga}\\ 
d^{p+q+1}_r(f \smile g) &= f \smile d^{q+1}_r(g), \label{E:dga'}\\ 
d^{p+1}_r(f) \smile g &= f \smile d^{q+1}_l(g). \label{E:dga''}
\end{align}
Using the coassociativity of~$\Csh$ and the compatibility of the product~$\mu$ with the actions $\rho$, $\lambda$ for our bimodule-algebra~$M$, one writes both sides of~\eqref{E:dga} as $\lambda (\Id_X \times \mu) (\Id_X \times f \times g) \Csh^{1,p,q}$. 
Similarly, both sides of~\eqref{E:dga'} and~\eqref{E:dga''} equal, respectively, 
$\rho (\mu \times \Id_X) (f \times g \times \Id_X) \Csh^{p,q,1}$ and $\mu (\rho \times \Id_M) (f \times \Id_X \times g) \Csh^{p,1,q}$.

As a consequence, $\smile$ induces an associative product on $H^* (X,\sigma;M)$. For trivial commutative coefficients, the commutativity of this induced product follows from Theorem~\ref{T:Cup0Cup1}. 
\end{proof} 

The theorem also admits a graphical proof, using the diagrammatic interpretation of the braided differentials (Figure~\ref{P:BrHom}) and of the quantum shuffle coproduct (Figure~\ref{P:Shuffle}).

\begin{remark}
The theorem, except for the graded commutativity statement, remains valid for YBE solutions in any preadditive monoidal category, with the same proof.
\end{remark}

\begin{example}\label{EX:FactorizationCup}
Let us resume the example of a monoid~$X$ with $\sigma(x,x') = (1,xx')$ and a linear $X$-bimodule~$M$. Additionally, we need an associative product~$\mu$ on~$M$ compatible with the $X$-actions. Basic examples are any ring~$M$ with trivial $X$-actions, or the linearization of~$X$ with the actions given by the multiplication on~$X$. The cup product restricts to the critical subcomplex $(CrC^k,d^k)$ (Example~\ref{EX:FactorizationDiff}). Indeed, if $\boldsymbol{w} \in X^{\times (p+q)}$ contains a~$1$, then so does $T_s (\boldsymbol{w})$ for any $s \in S_{p+q}$, and thus $f \smile g  = \sum_{s\in Sh_{p,q}} (-1)^{|s|}\mu(f \times g) T_{s^{-1}}$ is zero on such~$\boldsymbol{w}$ if both $f \in CrC^p$ and $g \in CrC^q$ vanish whenever one of their arguments is~$1$. Further, the very particular form of our braiding forces $T_s (\boldsymbol{w})$ to contain a~$1$ for any $\boldsymbol{w} \in X^{\times (p+q)}$ provided that $s \neq \Id$. So in $f \smile g$ all the terms but one vanish, yielding
\begin{align}\label{E:CupHoch}
f \smile g(x_1,\ldots, x_{p+q}) &= f(x_1,\ldots, x_{p})g(x_{p+1},\ldots, x_{p+q}).
\end{align}
One recognizes the classical cup product for the Hochschild cohomology.
\end{example}

\begin{remark}
The product $f \smile g$ can also be defined when one of $f$ and~$g$ takes values in~$X$: it suffices to replace~$\mu$ with $\lambda$ or~$\rho$. The differential of a map $f \colon X^{\times (k-1)} \to M$ can then be expressed as the graded commutator
\begin{align*}
d^{k}(f) &= \Id_X \smile f - (-1)^{k-1} f \smile \Id_X.
\end{align*} 
This directly implies the compatibility~\eqref{E:dgaTotal} between~$d^k$ and~$\smile$. 
\end{remark}

\begin{remark}\label{R:Dendriform}
The shuffle set $Sh_{p,q}$ decomposes into two parts, which comprise permutations $s \in Sh_{p,q}$ satisfying $1 = s(1)$ and $1 = s(p+1)$ respectively. This induces a decomposition $\Csh^{p,q} = \Csh^{p,q,\leftarrow} + \Csh^{p,q,\rightarrow}$ of the quantum shuffle coproduct, and hence a decomposition of the cup product on the cochain level: 
$$f \smile g = f \underset{\leftarrow}{\smile} g + f \underset{\rightarrow}{\smile} g.$$ 
Compatibility relations between $\Csh^{p,q,\leftarrow}$ and $\Csh^{p,q,\rightarrow}$ imply that $(\underset{\leftarrow}{\smile},\underset{\rightarrow}{\smile})$ is a graded \emph{dendriform algebra} structure\footnote{This structure was introduced by Loday under the name \emph{dual-dialgebra} \cite{LodayDiFr,LodayDi}.} on cochains. In general this decomposition does not survive in cohomology. However it does so when restricted to the subcomplex of maps $X^{\times k} \to M$ satisfying the symmetry condition $d^{k;1}_{l}f = d^{k;1}_{r}f$, or explicitly
$$x \cdot f(x_1, \ldots,x_k) = f(x'_1, \ldots,x'_k) \cdot x',$$
with $x'_1 \ldots x'_k x' = b_k \cdots b_1(x x_1 \ldots x_k)$. 
\end{remark}

\section{Circle product}\label{S:BrHomCircle}

Let a braided set $(X,\sigma)$ act trivially on a commutative ring~$\kk$. In order to finish the proof of Theorem~\ref{T:CupProd}, we will show that for coefficients in~$\kk$, the cup product on the cochain level is commutative up to an explicit homotopy. 

Take an integer $k>0$. On $C^k= \Map (X^{\times k},\kk)$, one has the bilinear commutative associative convolution product 
$$f \ast g(\boldsymbol{w}) = f(\boldsymbol{w}) g(\boldsymbol{w}).$$ 
Here the product in~$\kk$ is written as $(a,b) \mapsto ab$. For disjoint subsets $L$, $R$ of $\{1,2,\ldots,k\}$, let $\langle L,R \rangle$ be the number of couples $i \in L, j \in R$ with $i>j$. Let $i_1 < i_2 < \ldots < i_t$ be the properly ordered elements of $L \sqcup R$. Put $\omega_s = l$ if $i_s \in L$ and $\omega_s = r$ if $i_s \in R$. Consider the maps
\begin{align*}
d^{k;L,R} &= d^{k;i_t}_{\omega_t} \cdots d^{k-t+1;i_1}_{\omega_1} \colon C^{k-t} \to C^k.
\end{align*}
Finally, define the bilinear operation $\circ \colon C^{p} \times C^{q} \to C^{k}$, $k=p+q-1$, by
\begin{align}
f \circ g \, &= \,\sum (-1)^{(q-1) \# J_1 + \langle J_1,I_1\rangle + \langle I_2,J_2\rangle} d^{k;I_2,I_1}(f) \ast d^{k;J_1,J_2} (g),\label{E:Circ}
\end{align}
where the summation is over all $t \in \{1,\ldots,k\}$ and all decompositions $\{1,\ldots,t-1\} = I_1 \sqcup J_1$, $\{t+1,\ldots,k\} = I_2 \sqcup J_2$, with $\# I_1 + \# I_2 = q-1$, $\# J_1 + \# J_2 = p-1$. This operation is well defined. Its typical term is represented in Figure~\ref{P:Circle}. 
\begin{figure}[!h]\centering
\begin{tikzpicture}[xscale=0.65,yscale=0.35]
\draw [rounded corners, myred] (1,0)--(1,0.5)--(3,3)--(3,4)--(8.5,8)--(8.5,9);
\draw [rounded corners, myblue] (2,0)--(2,0.5)--(0,3.5)--(0,4.5)--(3,7.5)--(3,9);
\draw [rounded corners, myred] (3,0)--(3,0.5)--(4,3)--(4,4)--(9.5,8)--(9.5,9);
\draw [rounded corners, myblue] (4,0)--(4,0.5)--(1,3.5)--(1,4.5)--(4,7.5)--(4,9);
\draw [rounded corners] (5,0)--(5,9);
\draw [rounded corners, mygreen] (6,0)--(6,0.5)--(8,3.5)--(8,4.5)--(6,7.5)--(6,9);
\draw [rounded corners, mygreen,shift={(1,0)}] (6,0)--(6,0.5)--(8,3.5)--(8,4.5)--(6,7.5)--(6,9);
\draw [rounded corners, myviolet] (8,0)--(8,0.5)--(6,3.5)--(6,4.5)--(0,7.5)--(0,9);
\draw [dotted, thick] (2.5,4)--(6.5,4);
\draw [dotted, thick] (2.5,8)--(7.5,8);
\node at (2.5,4.4) {$\scriptstyle g$};
\node at (2.3,8.3) {$\scriptstyle f$};
\node at (8.5,9) [above, myred]{$\scriptstyle x''_1$};
\node at (3,9) [above, myblue]{$\scriptstyle x''_2$};
\node at (9.5,9) [above, myred]{$\scriptstyle x''_3$};
\node at (4,9) [above, myblue]{$\scriptstyle x''_4$};
\node at (5,9) [above]{$\scriptstyle x''_5$};
\node at (6,9) [above, mygreen]{$\scriptstyle x''_6$};
\node at (7,9) [above, mygreen]{$\scriptstyle x''_7$};
\node at (0,9) [above, myviolet]{$\scriptstyle x''_8$};
\node at (2.8,3.5) [right, myred]{$\scriptstyle x'_1$};
\node at (0.2,4) [left, myblue]{$\scriptstyle x'_2$};
\node at (3.8,3.5) [right, myred]{$\scriptstyle x'_3$};
\node at (1.2,4) [left, myblue]{$\scriptstyle x'_4$};
\node at (4.8,3.5) [right]{$\scriptstyle x'_5$};
\node at (2.8,3.5) [right, myred]{$\scriptstyle x'_1$};
\node at (7.8,4) [right, mygreen]{$\scriptstyle x'_6$};
\node at (2.8,3.5) [right, myred]{$\scriptstyle x'_1$};
\node at (8.8,4) [right, mygreen]{$\scriptstyle x'_7$};
\node at (5.9,3.5) [right, myviolet]{$\scriptstyle x'_8$};
\node at (1,0) [below, myred]{$\scriptstyle x_1$};
\node at (1,-1) [below, myred]{$\scriptstyle I_1$};
\node at (3,0) [below, myred]{$\scriptstyle x_3$};
\node at (3,-1) [below, myred]{$\scriptstyle I_1$};
\node at (2,0) [below, myblue]{$\scriptstyle x_2$};
\node at (2,-1) [below, myblue]{$\scriptstyle J_1$};
\node at (4,0) [below, myblue]{$\scriptstyle x_4$};
\node at (4,-1) [below, myblue]{$\scriptstyle J_1$};
\node at (5,0) [below]{$\scriptstyle x_5$};
\node at (6,0) [below, mygreen]{$\scriptstyle x_6$};
\node at (6,-1) [below, mygreen]{$\scriptstyle J_2$};
\node at (7,0) [below, mygreen]{$\scriptstyle x_7$};
\node at (7,-1) [below, mygreen]{$\scriptstyle J_2$};
\node at (8,0) [below, myviolet]{$\scriptstyle x_8$};
\node at (8,-1) [below, myviolet]{$\scriptstyle I_2$};
\end{tikzpicture} 
\caption{For $f \in C^5$, $g \in C^4$, $f \circ g$ is a signed sum containing the term $\pm d^{8;{\color{myviolet}\{8\}},{\color{myred}\{1,3\}}}(f) \ast d^{8;{\color{myblue}\{2,4\}},{\color{mygreen}\{6,7\}}} (g)(x_1,\ldots,x_8) = f(x''_2,x''_4,x''_5,x''_6,x''_7) g(x'_1,x'_3,x'_5,x'_8)$.}\label{P:Circle}
\end{figure}
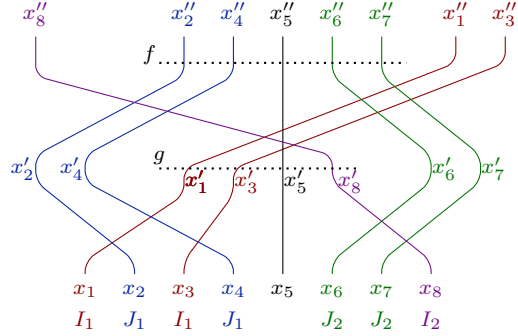
Here $f$ and~$g$ are evaluated on the arguments connected by the dotted lines. The $X$-strands preserve their color when passing through these dotted lines. The sign our term comes with contains two parts: the rearrangement of the arguments~$x_i$ is responsible for the Koszul sign $(-1)^{\langle J_1,I_1\rangle + \langle I_2,J_2\rangle}$, which can also be obtained from the crossing number of the part of the diagram below the $g$-line; the sign $(-1)^{(q-1) \# J_1}$ is produced when $g$ moves across the $\# J_1$ left $X$-strands in order to attain its arguments. The operation~$\circ$ should be compared with the dioperadic composition: see \cite{Gan} for the first mention, and \cite[Fig.~1]{KaWaZu} for the relevant diagrammatic version.

\begin{definition}
The operation~$\circ$ above is called the \emph{circle product}.
\end{definition}

\begin{example}\label{EX:FactorizationCicrle}
Let us resume our favorite example of a monoid~$X$, choosing trivial coefficients~$\kk$. An argument similar to that from Example~\ref{EX:FactorizationCup} shows that, for $f \in CrC^p$, $g \in CrC^q$, the terms of $f \circ g$ vanish except when $I_1 = \emptyset$ and $\langle L_2,R_2 \rangle=0$. The remaining terms are easy to write down explicitly:
\begin{align*}
f \circ g &= \sum_{t=1}^p (-1)^{(t-1)(q-1)} f(x_1,\ldots,x_{t-1}, x_t \cdots x_{t+q-1}, x_{t+q},\ldots,x_{p+q-1}) g(x_t, \ldots, x_{t+q-1}).
\end{align*}
This is the usual $\circ$-product for the group cohomology with trivial coefficients (which in this cases coincides with the Hochschild cohomology).
\end{example}

For small values of~$p$ or~$q$ the definition of~$\circ$ becomes less scary. For $f \in C^p$, $g \in C^1$, one computes
\begin{align}
f \circ g (x_1, \ldots,x_p) &= f(x_1, \ldots,x_p) \sum\nolimits_{i=1}^p g(x_i),\label{E:CircDeg1}\\
g \circ f (x_1, \ldots,x_p) &= \sum\nolimits_{i=1}^p g(x''_i)f(x_1, \ldots,x_p),\label{E:CircDeg1'}
\end{align} 
where $(x''_{p}, \ldots, x''_{1}) = \Delta_p (x_{1}, \ldots, x_{p})$. The element $\Delta_p \in B^+_p$ is defined by~\eqref{E:Delta}, and as usual acts on~$X^{\times p}$ via the braiding~$\sigma$. If $g$ is a cocycle, then the definition~\eqref{E:d2} of $d^2$ and the commutativity of~$\kk$ imply $f \circ g = g \circ f$.

For $f \in C^p$, $g \in C^2$, one computes
\begin{align*}
f \circ g (x_1, \ldots,x_{p+1}) = \sum\nolimits_{i=1}^{p+1}(-1)^{i}(&f(x'_1,\ldots,x'_{i-1},x_{i+1},\ldots,x_{p+1})g^l(x_1,\ldots,x_i)\\
&-f(x_1,\ldots,x_{i-1},x''_{i+1},\ldots,x''_{p+1})g^r(x_i,\ldots,x_{p+1})),
\end{align*}
\begin{align*}
g^l(x_1,\ldots,x_i)& = g(x_{i-1},x_i) + g(x_{i-2},x_i^{l,2})+\cdots + g(x_{1},x_i^{l,i-1}),\\
g^r(x_i,\ldots,x_{p+1}) &= g(x_{i},x_{i+1}) + g(x_i^{r,2},x_{i+2}) +\cdots + g(x_i^{r,p+1-i},x_{p+1}), 
\end{align*} 
where $\sigma(x_{i-1},x_i) = (x_i^{l,2},x'_{i-1})$, $\sigma(x_{i-2},x_i^{l,2}) = (x_i^{l,3},x'_{i-2})$ etc., and similarly for the right counterparts $x_i^{r,j}$ and $x''_{j}$. Recalling the definition of braided differentials, one recognizes in $-f \circ g$ the terms $\pm d^{p+1;i}_{\omega}f$ of $d^{p+1}f$, $\omega \in \{l,r\}$, taken with the $g^{\omega}$-weights. These weights are the sums of the evaluations of~$g$ on all the crossings of the diagram representing $d^{p+1;i}_{\omega}f$. The example from Figure~\ref{P:CircleForq2} should clarify this description. 
\begin{figure}[!h]\centering
\begin{tikzpicture}[xscale=0.7,yscale=0.5]
 \draw [thick, mygreen] (4,-1.8) --  (4,0);
 \draw (-1,-3) rectangle (9,-1.8);
 \node at (4,-2.4) {$\scriptstyle f$}; 
 \node at (-.4,-4.6) {$\scriptstyle g$};  
 \draw [dotted, thick] (-.3,-4.3)--(1.2,-4.3);
 \node at (2.6,-6.1) {$\scriptstyle g$};  
 \draw [dotted, thick] (2.7,-5.8)--(4.2,-5.8); 
 \node at (4.1,-6.6) {$\scriptstyle g$};  
 \draw [dotted, thick] (4.2,-6.6)--(5.7,-6.6);  
 \draw [red,rounded corners] (4,-0.5) -- (-2,-1.5) -- (-2,-3) -- (6,-7) -- (6,-8);
 \draw (0,-3) --  (0,-8);
 \draw (3,-3) -- (3,-8);
 \draw (4.5,-3) -- (4.5,-8);
 \draw (8,-3) -- (8,-8);
 \node at (-2.3,-2.4) [myblue] {$\scriptstyle \boldsymbol{x'_i}$};
 \node at (0,-8.5) [myblue] {$\scriptstyle \boldsymbol{x_1}$};
 \node at (3,-8.5) [myblue] {$\scriptstyle \boldsymbol{x_{i-2}}$}; 
 \node at (1.3,-5.3) [myblue] {$\scriptstyle \boldsymbol{x_i^{l,i-1}}$}; 
 \node at (3.75,-7.6) [myblue] {$\scriptstyle \boldsymbol{x_i^{l,2}}$};
 \draw [thick,rounded corners,-latex,myblue] (3.6,-7.4) -- (3.55,-6.6) -- (4,-6); 
 \node at (4.5,-8.5) [myblue] {$\scriptstyle \boldsymbol{x_{i-1}}$};
 \node at (6,-8.5) [myblue] {$\scriptstyle \boldsymbol{x_i}$}; 
 \node at (8,-8.5) [myblue] {$\scriptstyle \boldsymbol{x_{p+1}}$};
 \node at (0.3,-3.4) [myblue] {$\scriptstyle \boldsymbol{x'_1}$}; 
 \node at (3.5,-3.5) [myblue] {$\scriptstyle \boldsymbol{x'_{i-2}}$};
 \node at (5,-3.5) [myblue] {$\scriptstyle \boldsymbol{x'_{i-1}}$};  
\end{tikzpicture}
   \caption{The $g^l$-weight modifying $d^{p+1;i}_{l}f$ in the computation of $f \circ g$, $g \in C^2$.}\label{P:CircleForq2}
\end{figure}
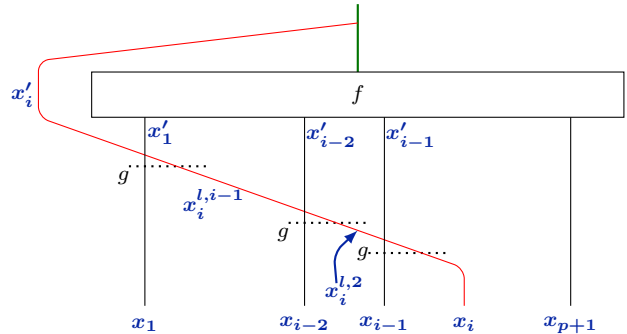

\begin{theorem}\label{T:Cup0Cup1}
Let a braided set $(X,\sigma)$ act trivially on a commutative ring~$\kk$. The circle product measures the commutativity defect of the cup product for braided cochains of $(X,\sigma)$ with coefficients in~$\kk$. Concretely, for maps $f \colon X^{\times p} \to \kk$, $g \colon X^{\times q} \to \kk$, one has
\begin{align}
d^{p+q}(f \circ g) &- (-1)^{q-1}(d^{p+1} f) \circ g - f \circ (d^{q+1} g) = \notag\\
 &(-1)^{q}( g \smile f - (-1)^{pq} f \smile g).\label{E:CircCup}
\end{align}
\end{theorem}

The graded commutativity in cohomology follows.

\begin{proof}
By definitions, $d^{p+q}(f \circ g)$ is a signed sum of terms of the form $d^{k+1;i}_{\omega}(d^{k;I_2,I_1}(f) \ast d^{k;J_1,J_2} (g))$, $\omega \in \{l,r\}$, $k = p+q-1$. Such a term is called \emph{initial} if $i \le \# I_1 + \# J_1+1$, and \emph{final} otherwise. Similarly, a term $d^{k+1;I_2,I_1}(d^{p+1;i}_{\omega} f) \ast d^{k+1;J_1,J_2} (g)$ of $(d^{p+1} f) \circ g$ is \emph{initial} if $i \le \# J_1$, \emph{middle} if $i = \# J_1+1$, and \emph{final} otherwise. A term $d^{k+1;I_2,I_1}(f) \ast d^{k+1;J_1,J_2} (d^{q+1;i}_{\omega} g)$ of $f \circ (d^{q+1} g)$ is \emph{initial} if $i \le \# I_1$, \emph{middle} if $i = \# I_1+1$, and \emph{final} otherwise. As usual, such terms are declared \emph{left} or \emph{right} depending on $\omega \in \{l,r\}$.

Now, on the left-hand side of~\eqref{E:CircCup}, most terms annihilate each other, namely
\begin{itemize}
\item the left initial terms of $d^{p+q}(f \circ g)$ and the left initial terms of $(d^{p+1} f) \circ g$;
\item the left final terms of $d^{p+q}(f \circ g)$ and the left final terms of $f \circ (d^{q+1} g)$;
\item the right initial terms of $d^{p+q}(f \circ g)$ and the right initial terms of $f \circ (d^{q+1} g)$;
\item the right final terms of $d^{p+q}(f \circ g)$ and the right final terms of $(d^{p+1} f) \circ g$;
\item the right initial terms of $(d^{p+1} f) \circ g$ and the left initial terms of $f \circ (d^{q+1} g)$;
\item the left final terms of $(d^{p+1} f) \circ g$ and the right final terms of $f \circ (d^{q+1} g)$.
\end{itemize}
A possible verification combines diagrammatic reasoning with careful sign book-keeping. Thus only the middle terms remain. They all have the form $\pm d^{k+1;I_2,I_1}(f) \ast d^{k+1;J_1,J_2} (g)$ for some $t \in \{0,\ldots,k+1\}$ and some decompositions $\{1,\ldots,t\} = I_1 \sqcup J_1$, $\{t+1,\ldots,k+1\} = I_2 \sqcup J_2$; such terms will be called \emph{complete}. The middle terms from $(d^{p+1}_r f) \circ g$ and  $f \circ (d^{q+1}_l g)$ yield all the complete terms with $t \in I_1$ and $t \in J_1$ respectively. Similarly,  the middle terms from $(d^{p+1}_l f) \circ g$ and  $f \circ (d^{q+1}_r g)$ yield all the complete terms with $t+1 \in I_2$ and $t+1 \in J_2$. Thus each complete term with $t \neq 0, k+1$ appears on the left-hand side of~\eqref{E:CircCup} twice, with opposite signs.  The complete terms with $t=0$ sum up to $(-1)^{q}g \smile f$, and those with $t=k+1$ to $(-1)^{pq+q+1} f \smile g$. One obtains precisely the right-hand side of~\eqref{E:CircCup}.
\end{proof}

It would be interesting to know if a weaker form of~\eqref{E:CircCup} and the resulting commutativity of the cup product in cohomology remain valid for more general coefficients. For instance, observe that the definition of~$\circ$ remains valid for any bimodule-algebra as coefficients. Further, the commutativity of~$\kk$ is used only at the end of the proof above; without it the right-hand side of~\eqref{E:CircCup} should be replaced with 
$$(-1)^{q}( \mu \tau (g \times f) \Csh^{q,p} - (-1)^{pq} f \smile g),$$
where $\mu$ is the product on~$\kk$, and $\tau$ is the flip $(a,b) \mapsto (b,a)$.

It is also natural to ask if our operations $\smile, \circ$ satisfy other properties of \emph{homotopy Gerstenhaber algebras}, as is the case for the simplicial cohomology, or for the Hochschild cohomology with coefficient in the monoid~$X$ itself. (See~\cite{HGerst} for the original definitions, and~\cite{HGerstKade} for a recent concise overview.) The answer is negative. For example, the Hirsch formula
$$(f \smile g) \circ h = f \smile (g \circ h) + (-1)^{|g|(|h|-1)} (f \circ h) \smile g,$$
where we write $|g| = q$ for $g \in C^q$, fails even for $f,g,h \in C^1$, unless $h$ is a cocycle. This is easily checked using formulas \eqref{E:CircDeg1}-\eqref{E:CircDeg1'}. Similarly, the pre-Lie condition
$$(f \circ g) \circ h - f \circ (g \circ h) = (-1)^{(|g|-1)(|h|-1)} ((f \circ h) \circ g - f \circ (h \circ g))$$
generally fails for $h \in C^1$, this time even when it is a constant cocycle!

\begin{remark}
As our definition of the circle product suggests, it generalizes verbatim to any pre-cubical set. Furthermore, for trivial coefficients the cup product can be computed by the formula
\begin{align}\label{E:CupCub}
f \smile g \, &= \,\sum (-1)^{\langle I,J\rangle} d^{p+q;\emptyset, J}(f) \ast d^{p+q;I,\emptyset} (g),
\end{align}
with the sum running over all decompositions $\{1,\ldots,p+q\} = I \sqcup J$, $\# I=p$. This definition, as well as our proof of relation~\eqref{E:CircCup}, work well in the pre-cubical setting. We thus recover the cup product for the pre-cubical cohomology, known already to Serre~\cite[Section II.1]{SerreThesis}. He deduced its graded commutativity from that of the cup product for the more classical pre-simplicial cohomology, the two theories being related by the Eilenberg--Zilber map. A detailed exposition of this approach, with explicit formulas, was given by Clauwens~\cite{Clauwens}. He also developed applications to the cohomology of self-distributive structures, which is a particular case of braided, and thus pre-cubical, cohomology \cite{HomologyYB,Lebed1}. In the self-distributive case, Covez~\cite{Covez} refined the commutative structure in cohomology into a Leibniz one, with the help of the decomposition from Remark~\ref{R:Dendriform}. The circle product in the pre-cubical setting is due to Baues~\cite{Baues}. Kadeishvili~\cite{KadeishviliDGHopf} included it into an infinite series of Steenrod-like operations $\smile_i \colon C^p \otimes C^q \to C^{p+q-i}$, $\smile_0 = \smile$, $\smile_1 = \circ$, such that each~$\smile_i$ is graded (anti)\-commutative 
 up to a homotopy given by~$\smile_{i+1}$. This is compatible with the property $f \circ g = g \circ f$ we established for a cocycle $g \in C^1$. In spite of this general theory, we presented here detailed constructions and proofs for the particular case of braided cohomology, for several reasons. First, our constructions are given by concise explicit formulas and, alternatively, by handy diagrammatic calculus, unavailable in the general situation. Second, our definition~\eqref{E:Cup} of the cup product differs from and better behaves than the pre-cubical definition~\eqref{E:CupCub} when the coefficients are not trivial. For instance, for the Hochschild cohomology, the latter yields the operation
 $$f \smile g(x_1,\ldots, x_{p+q}) = (f(x_1,\ldots, x_{p}) \cdot x_{p+1} \cdots x_{p+q})(x_1 \cdots x_{p} \cdot g(x_{p+1},\ldots, x_{p+q}))$$
instead of the usual Hochschild cup product.
\end{remark}

\section{Braided (co)homology for idempotent braidings}\label{S:BrHomSG}

This section describes certain subcomplexes and quotients of the braided (co)chain complexes for $(X,\sigma)$. For an idempotent $\sigma$, they are shown to compute the (co)homology of the structure monoid $\Mon(X,\sigma)$, while being significantly smaller than the complexes given by the bar resolution. Applications to the computation of the (co)homology of factorizable monoids are given here. Similar applications for plactic monoids are described in~\cite{LebedPlactic}. To improve the readability, we postponed the rather technical proofs of the results of this section until Section~\ref{S:Proof}.

Recall the braided (co)chain complexes from Theorem~\ref{T:BrHom}, and the cup and circle products on cochains defined by~\eqref{E:Cup} and~\eqref{E:Circ}.

\begin{proposition}\label{PR:InvarSubcx}
Let $(X,\sigma)$ be a braided set, and $R$ a sub-group of $\ZZ X^{\times 2}$ such that
\begin{enumerate}[label=\Alph*.]
\item\label{I:Fixed} (the linearization of) $\sigma$ restricts to the identity on~$R$; 
\item\label{I:Invar} $R$ is a sub-$(X,\sigma)$-bimodule of $\ZZ X^{\times 2}$ for the adjoint actions (Example~\ref{EX:Adjoint});
\item\label{I:Invar'} for any $\boldsymbol{r} \in R$ written as $\sum_i \alpha_i (x_i,y_i)$, $\alpha_i \in \ZZ\setminus \{0\}$, $x_i,y_i \in X$, and for any $x \in X$, the left/right adjoint actions of the $x_iy_i$ on~$x$ coincide for all~$i$.
\end{enumerate}
Denote by $T(X;R)$ the two-sided ideal $\ZZ\langle X \rangle R \ZZ\langle X \rangle$ of $\ZZ\langle X \rangle$, and by $T_k(X;R)$ its degree~$k$ component. Then for the same type of coefficients $M,N$ as in Theorem~\ref{T:BrHom},
\begin{enumerate}
\item $M \otimes T(X;R) \otimes N$ is a subcomplex of $(M \otimes_{\ZZ} \ZZ X^{\times k} \otimes_{\ZZ} N, d_k)$;
\item $M \otimes T(X;R)$ is a subcomplex of $(M \otimes_{\ZZ} \ZZ X^{\times k}, d_k)$;
\item the maps $X^{\times k} \to M$ whose linearization vanishes on $T_k(X;R)$ form a subcomplex of $(\Map (X^{\times k},M), d^k)$, closed under the cup and circle products if $M$ is a bimodule-algebra. 
\end{enumerate}
\end{proposition}

\begin{notation}
The quotients by the subcomplexes from Points~1 and~2 are denoted by $C_*(X; M,N;R)$ and $C_*(X; M;R)$. Notation $C^*(X; M; R)$ is used  for the subcomplex from the last point.
\end{notation}

We now give two examples of sub-groups $R$ of $\ZZ X^{\times 2}$ satisfying the required conditions. The quotients and subcomplexes associated to these~$R$ were considered, from a different perspective, by Farinati and Garc{\'{\i}}a-Galofre \cite{FarinatiGalofre}.

First, let $R_{+}$ be the subgroup generated by $(x,y)+\sigma(x,y)$ for all $x,y \in X$. It always satisfies conditions \ref{I:Invar}-\ref{I:Invar'}, while \ref{I:Fixed} is equivalent to the involutivity of~$\sigma$. Thus for involutive~$\sigma$ our proposition applies to~$R_{+}$. Taking as coefficients the structure monoid $M=\Mon(X,\sigma)$ (Example~\ref{EX:StructureSGAsBimodAlg}), Farinati and Garc{\'{\i}}a-Galofre showed the complex $(\kk \otimes_{\ZZ} C_*(X; M,M;R_{+}),d_k)$ to give a free resolution of the algebra $\kk M$ viewed as a bimodule over itself, at least in characteristic zero. This resolution is much smaller than the bar resolution---among others since $\Mon(X,\sigma)$ is always infinite, even for finite~$X$.

Now, let $R_{-}$ be the subgroup generated by $\sigma$-invariant pairs $(x,y)$ (in the sense of $\sigma(x,y) = (x,y)$). Conditions \ref{I:Fixed}-\ref{I:Invar'} are automatic here, so our proposition applies. 

\begin{definition}
For $R = R_{-}$, the above complexes and their (co)homology are called the \emph{critical\footnote{The term \emph{normalized} is more usual for such quotients and sub-complexes. However, it would be in conflict with the normalization terminology: critical chains are precisely those where no subsequent elements form a normal word!} complexes / (co)homology} of $(X,\sigma)$. They are denoted by $CrC^*(X,\sigma; M)$, $CrH^*(X,\sigma; M)$ etc.
\end{definition}

\begin{definition}
Let $(X,\sigma)$ be a braided set. A word $x_1 \ldots x_k \in X^*$ is called \emph{critical} if $\sigma(x_i,x_{i+1}) \neq (x_i,x_{i+1})$ for all~$i$. The set of such words is denoted by $Cr_k(X,\sigma)$.
\end{definition}

Normal and critical words constitute two extreme word types in~$X^*$.

One has obvious decompositions $\ZZ X^{\times k} = T_k(X;R_{-}) \oplus \ZZ Cr_k(X,\sigma)$. It implies identifications of type 
$$CrC_*(X,\sigma; M,N) \simeq \bigoplus_{k \ge 0} M \otimes_{\ZZ} \ZZ Cr_k(X,\sigma) \otimes_{\ZZ} N.$$

In the pseudo-unital case, braided (co)chain complexes can be reduced even further:

\begin{proposition}\label{PR:UnitSubcx}
Let $(X,\sigma,1)$ be a PUIBS. Denote by $T(X,\sigma,1)$ the two-sided ideal $\ZZ\langle X \rangle (R_{-} + \ZZ 1) \ZZ\langle X \rangle$ of $\ZZ\langle X \rangle$, and by $T_k(X,\sigma,1)$ its degree~$k$ component. The assertions of Proposition~\ref{PR:InvarSubcx} remain valid with $T_{(k)}(X,\sigma,1)$ replacing $T_{(k)}(X;R)$ everywhere.
\end{proposition}

\begin{notation}
Notations $CrC^*(X,\sigma,1; M)$, $CrH^*(X,\sigma,1; M)$ etc. are used for the above complexes and their (co)homology. They are called the \emph{critical complexes / (co)homology} of $(X,\sigma,1)$. The \emph{critical words} are not supposed to contain the letter~$1$ for a PUIBS; their set is denoted by $Cr_k(X,\sigma,1)$.
\end{notation}

Again, decompositions $\ZZ X^{\times k} = T_k(X,\sigma,1) \oplus \ZZ Cr_k(X,\sigma,1)$ imply 
$$CrC_*(X,\sigma,1; M,N) \simeq \bigoplus_{k \ge 0} M \otimes_{\ZZ} \ZZ Cr_k(X,\sigma,1) \otimes_{\ZZ} N.$$

We now turn to a comparison between the braided (co)homology of a braided set and the Hochschild (co)homology of its structure monoid.

\begin{theorem}\label{T:TotalSh}
\begin{enumerate}
\item Let $(X,\sigma)$ be a braided set. Consider the linear maps
\begin{align*}
\TotSh_k \colon \qquad \ZZ X^{\times k} &\longrightarrow \ZZ X^{\times k} \hookrightarrow \ZZ \Mon(X,\sigma)^{\times k},\\
x_1 x_2 \cdots x_k & \longmapsto x_1 \Sh x_2 \Sh \cdots \Sh x_k.
\end{align*}
In all the three situations from Theorem~\ref{T:BrHom}, they yield morphisms between braided (co)\-chain complexes for $(X,\sigma)$ and Hochschild (co)chain complexes for $\Mon(X,\sigma)$ with the same coefficients. Moreover, $\TotSh_k$ vanishes on the non-critical part $T_k(X;R_{-})$ of $\ZZ X^{\times k}$.

\item If $(X,\sigma,1)$ is a PUIBS, consider the composition $\oTotSh_k$ of $\TotSh_k$, the projection $\ZZ \Mon(X,\sigma)^{\times k}$ $\twoheadrightarrow \ZZ \oMon(X,\sigma,1)^{\times k}$, and the map sending all $k$-tuples containing at least one empty word to~$0$. It yields morphisms between braided (co)chain complexes for $(X,\sigma)$ and normalized Hochschild (co)chain complexes for $\oMon(X,\sigma,1)$. Moreover, $\oTotSh_k$ vanishes on the non-critical part $T_k(X,\sigma,1)$ of $\ZZ X^{\times k}$.

\item In the case of cochain complexes with coefficients in a bimodule-algebra, the maps above preserve cup products.
\end{enumerate}
\end{theorem}

Points~1 and~3 are due to Farinati and Garc{\'{\i}}a-Galofre \cite{FarinatiGalofre}. For braidings associated to racks, a related result was obtained by Covez \cite[Section~5]{Covez}.

\begin{definition}
The maps from the proposition are called the \emph{quantum symmetrizer} and the \emph{reduced quantum symmetrizer} respectively.
\end{definition}

\begin{remark}
The quantum symmetrizer does not preserve circle products in general. Indeed, for a $2$-cochain $f$ and a $1$-cochain $g$ of $\Mon(X,\sigma)$, one calculates
$$\big( \TotSh_2^*(f \circ g) - \TotSh_2^*(f)\circ \TotSh_1^*(g)\big) (x,y) = f(y',x') d^2g(x,y)$$
for all $x,y \in X$, $(y',x') = \sigma(x,y)$.
\end{remark}

It is natural to ask if one can transform the (reduced) quantum symmetrizer into a quasi-isomorphism. Related questions recently appeared in~\cite{FarinatiGalofre} and~\cite{YangGraphYBE}. Theorem~\ref{T:TotalSh} suggests that one should restrict the braided (co)homology to its critical part. We now show that for idempotent braidings this is sufficient.

\begin{theorem}\label{T:BrHomIdempot}
Let $M$ be a linear bimodule over an idempotent braided set $(X,\sigma)$. 
We also regard it as a bimodule over the structure monoid $\Mon(X,\sigma)$. The quantum symmetrizer then induces quasi-isomorphisms between the critical (co)chain complex for $(X,\sigma)$ and the Hochschild (co)chain complex for $\Mon(X,\sigma)$, both with coefficients in~$M$: 

\begin{align*}
\TotSh_k \colon CrC_k(X,\sigma;M) &\overset{\text{q-iso}}{\longrightarrow} HC_k(\Mon(X,\sigma);M),\\
CrC^k(X,\sigma;M) &\overset{\text{q-iso}}{\longleftarrow} HC^k(\Mon(X,\sigma);M) \colon \TotSh^k_*.
\end{align*}
If moreover $(X,\sigma)$ admits a pseudo-unit~$1$ acting on~$M$ trivially, then the critical (co)chain complexes for $(X,\sigma,1)$ compute the Hochschild (co)homology of the reduced structure monoid:
\begin{align*}
\oTotSh_k \colon CrC_k(X,\sigma,1;M) &\overset{\text{q-iso}}{\longrightarrow} HC_k(\oMon(X,\sigma,1);M),\\
CrC^k(X,\sigma,1;M) &\overset{\text{q-iso}}{\longleftarrow} HC^k(\oMon(X,\sigma,1);M) \colon \oTotSh^k_*.
\end{align*}
\end{theorem}

As a consequence, one gets linear graded isomorphisms in (co)homology. According to Theorem~\ref{T:TotalSh}, in the cohomological case they are algebra isomorphism when the cup products are defined. Braided techniques thus allow computations of the Hochschild (co)homo\-logy of structure monoids. We illustrate their efficiency by examples.

Let us start with an elementary warm-up example.

\begin{example}\label{EX:1element}
Take the set $X = \{x\}$ with $\sigma(x,x) = (x,x)$. For these data, the structure monoid $\Mon(X,\sigma)$ is freely generated by~$x$, and there are only two critical words: $x$ and the empty word~$\emptyw$. Take an abelian group~$M$ with trivial $X$-actions. The complex $CrC_k(X,\sigma;M)$ then reads
$$0 \overset{0}{\longleftarrow} M \overset{0}{\longleftarrow} M \otimes \ZZ \{x\} \overset{0}{\longleftarrow} 0 \overset{0}{\longleftarrow} 0 \overset{0}{\longleftarrow} \ldots$$
Its homology is immediate to compute: it is $M$ in degrees~$0$ and~$1$, and $0$ elsewhere. Due to Theorem~\ref{T:BrHomIdempot}, this yields the homology of the much larger Hochschild complex
$$0 \overset{0}{\longleftarrow} M \overset{d_1^{H}}{\longleftarrow} M \otimes \ZZ [x] \overset{d_2^{H}}{\longleftarrow} M \otimes \ZZ[x]^{\otimes 2} \overset{d_3^{H}}{\longleftarrow} M \otimes \ZZ [x]^{\otimes 3} \overset{d_4^{H}}{\longleftarrow} \ldots$$
The quantum symmetrizer is given by the identity in degree~$0$, and by the linearization of $x \mapsto x$ in degree~$1$. Similar computations in cohomology yield $H^k = M$ for $k=0,1$, and $H^k = 0$ for $k>1$. Further, $(X, \sigma, x)$ is a PUIBS. The reduced structure monoid $\oMon(X,\sigma,x)$ has one element only. The only critical word for $(X, \sigma, x)$ is~$\emptyw$. Thus, for the same coefficients, the critical complex for $(X,\sigma,x)$ coincides with the normalized Hochschild complex for the one-element monoid, and reads
$$0 \overset{0}{\longleftarrow} M \overset{0}{\longleftarrow} 0 \overset{0}{\longleftarrow} 0 \overset{0}{\longleftarrow} \ldots$$
\end{example}

Next, we recover the classical small resolutions of free and symmetric algebras.

\begin{example}
Generalizing Example~\ref{EX:1element}, take any set~$X$ with $\sigma(x,y) = (x,y)$. Here $\Mon(X,\sigma)$ is the monoid $\langle X \rangle$ freely generated by~$X$, and $X \sqcup \{\emptyw\}$ is the set of critical words. Take a linear $X$-bimodule~$M$. The complex $CrC_k(X,\sigma;M)$ reads
$$0 \overset{0}{\longleftarrow} M \overset{d_1}{\longleftarrow} M \otimes \ZZ X \overset{0}{\longleftarrow} 0 \overset{0}{\longleftarrow} 0 \overset{0}{\longleftarrow} \ldots$$
where $d_1(m,x) = m \cdot x - x \cdot m$. Due to Theorem~\ref{T:BrHomIdempot}, this describes the Hochschild homology of $\langle X \rangle$ (equivalently, of $\ZZ \Mon(X,\sigma) = T(\ZZ X)$) as 
\begin{align*}
HH_0(\langle X \rangle;M) &= \raisebox{.1cm}{$M$} / \raisebox{-.1cm}{$m \cdot x - x \cdot m$},\\
HH_1(\langle X \rangle;M) &= \{\, (m_x)_{x \in X} \, | \, m_x \in M, \, \sum_{x \in X} (m_x \cdot x - x \cdot m_x) = 0 \, \},\\ 
HH_k(\langle X \rangle;M) &= 0 \hspace*{1cm} \text{ for } k > 1.
\end{align*}
Above only a finite number of~$m_x$ is allowed to be non-zero.
\end{example}

\begin{example}
As explained in Example~\ref{EX:Bubble}, the symmetric monoid $S(X)$ of a set~$X$, which we endow with an arbitrary total order, is the structure monoid of~$X$ with the idempotent braiding 
$\sigma(x,y) = (\, \min \{x,y \} \, , \, \max \{x,y \} \,)$. 
Here the set of critical words of length~$k$ coincides with
$$\Lambda_k(X) = \{\, x_1 \ldots x_k \, |\, x_i \in X, \, x_1 > \cdots > x_k \, \}.$$
For such words, the braided differential reads
\begin{align*}
d_k(m,x_1 \ldots x_k) &= \sum_{i=1}^k (-1)^{i-1} (m \cdot x_i - x_i \cdot m,x_1, \ldots, x_{i-1}, x_{i+1}, \ldots, x_k).
\end{align*}
According to Theorem~\ref{T:BrHomIdempot}, this complex computes $HH_*(S(X);M) = HH_*(\ZZ [X];M)$ (since $\ZZ S(X) = \ZZ [X]$). In particular, if the cardinality $\# X$ is finite, then this homology vanishes in degree $> \# X$. Further, if the bimodule~$M$ is symmetric, then the braided differentials are all zero, and one concludes
\begin{align*}
HH_*(S(X);M) &= HH_*(\ZZ [X];M) = M \otimes \ZZ \Lambda_*(X),\\
HH^*(S(X);M) &= HH^*(\ZZ [X];M) = \Map (\Lambda_*(X),M).
\end{align*}
If $M$ comes with a bilinear product compatible with the $X$-actions, then the cup product for critical cochains can be computed by 
$$f \smile g (x_1, \ldots, x_{p+q}) = \sum_{ \{1, \ldots, p+q\} = I \sqcup J, \# I = p } (-1)^{\langle I,J\rangle} f(x_{i_1},\ldots,x_{i_p}) g(x_{j_1},\ldots,x_{j_q}),$$  
where $f$ and $g$ are cochains of degrees $p$ and $q$ respectively.
\end{example} 

We finish with our main example, where we obtain original results.

\begin{example}
Recall that a factorized monoid $G=HK$ can be regarded as the reduced structure monoid $\oMon(X,\sigma,1)$, where $X = H \cup K$, and the idempotent braiding~$\sigma$ is described in Example~\ref{EX:Factorization}. The critical words are describes here by
$$Cr_k(X,\sigma,1) = \sqcup_{p+q=k} \overline{K}^{\times p} \times \overline{H}^{\times q},$$
where $\overline{K} = K \setminus \{1\}$, $\overline{H} =H \setminus \{1\}$. Let $M$ be a linear $G$-bimodule. Theorem~\ref{T:BrHomIdempot} then identifies the Hochschild homology of~$G$ with the total homology of the double complex $(M \otimes \ZZ (\overline{K}^{\times p} \times \overline{H}^{\times q}), d_{p,q}^v, d_{p,q}^h)$, where
\begin{align*}
d_{p,q}^v(m&,k_1, \ldots, k_p,h_1, \ldots, h_q) = (m \cdot k_1, k_2, \ldots, k_p,h_1, \ldots, h_q)\\
&+\sum_{i=1}^{p-1} (-1)^i (m, k_1, \ldots, k_{i-1}, k_{i}k_{i+1}, k_{i+2}, \ldots, k_p,h_1, \ldots, h_q)\\
&+ (-1)^{p}(k'_p \cdot m,k_1, \ldots, k_{p-1},h'_1, \ldots, h'_q),\\
& \qquad\qquad h'_1 \ldots  h'_q k'_p =b_q \cdots b_1 (k_p h_1 \ldots h_q);\\
d_{p,q}^h(m&,k_1, \ldots, k_p,h_1, \ldots, h_q) = (m \cdot h''_1,k''_1, \ldots, k''_{p-1}, k''_p,h_2, \ldots, h_q)\\
&+\sum_{i=1}^{q-1} (-1)^i (m,k_1, \ldots, k_p,h_1, \ldots, h_{i-1}, h_{i}h_{i+1}, h_{i+2}, \ldots, h_q)\\
&+ (-1)^{q}(h_q \cdot m,k_1, \ldots, k_p,h_1, \ldots, h_{q-1}),\\
& \qquad\qquad h''_1 k''_1 \ldots k''_p = b_1 \cdots b_p (k_1 \ldots k_p h_1).
\end{align*}
In all formulas from this example, the terms containing the element~$1$ are omitted. This total complex is much smaller than the one given by the bar resolution. Moreover, it has more structure, which allows for instance to apply the spectral sequence machinery for computations. In the case of a direct product $G=H \times K$, one has $\sigma(k,h) = (h,k)$ for all $k \in K, h \in H$; the formulas above then simplify, and, for trivial coefficients, recover the K\"{u}nneth formula. 
The cohomology of a factorized monoid is computed by a similar double complex. Suppose now that $M$ comes with a bilinear product compatible with the $G$-actions, in the sense of $g \cdot (m_1 m_2) = (g \cdot m_1) m_2$, $(m_1 m_2) \cdot g = m_1 (m_2 \cdot g)$, $(m_1 \cdot g) m_2 = m_1 (g \cdot m_2)$ for all $m_1,m_2 \in M, \, g \in G$. Then the cohomology groups of~$G$ with coefficients in~$M$ carry the cup product, which on the level of critical cochains corresponds to
\begin{align*}
(f \smile g) (k_1, \ldots, k_p,h_1, &\ldots, h_q) = \\
\sum_r (-1)^{(p-r)(s-r)}& f (k_1, \ldots, k_r,h'_1, \ldots, h'_{s-r}) g (k'_{r+1}, \ldots, k'_p,h_{s-r+1}, \ldots, h_q),\\
&h'_1 \ldots h'_{s-r} k'_{r+1} \ldots k'_p = b_{p-r,s-r}(k_{r+1} \ldots k_p h_1 \ldots h_{s-r}).
\end{align*}
Here $f$ is an $s$-cochain, and $g$ is a $t$-cochain, with $s+t = p+q$; the element $b_{p-r,s-r}$ of the Coxeter monoid $C_{p+s-2r}$ is defined in Figure~\ref{P:YBE}\rcircled{B}. Indeed, the remaining terms of $\Csh^{s,t}(k_1, \ldots, k_p,h_1, \ldots, h_q)$ necessarily contain the element~$1$.
\end{example}

\section{Proofs}\label{S:Proof}

\subsection*{Proof of Proposition~\ref{PR:InvarSubcx}}

We will show that $T(X;R)$ is a co-ideal of $\ZZ\langle X \rangle$ for the shuffle co-product~$\Csh$ from~\eqref{E:qu_cosh}. This directly implies most of the statements.

For some $\boldsymbol{w} \in X^{\times i}$, $\boldsymbol{v} \in X^{\times j}$, and $\boldsymbol{r} \in R$, consider the term $(-1)^{|s|} T_{s^{-1}} \boldsymbol{w}\boldsymbol{r}\boldsymbol{v}$ of $\Csh^{p,q} \boldsymbol{w}\boldsymbol{r}\boldsymbol{v}$, with $s\in Sh_{p,q}$, $p+q = i+j+2$. If $s^{-1}$ sends both $i+1$ and $i+2$ to $\{1,\ldots, p\}$, then conditions \ref{I:Invar}-\ref{I:Invar'} imply $T_{s^{-1}} \boldsymbol{w}\boldsymbol{r}\boldsymbol{v} \in T_p(X;R) \otimes_{\ZZ} \ZZ X^{\times q}$. The case of $s^{-1}$ sending $i+1$ and $i+2$ to $\{p+1,\ldots, p+q\}$ is similar. Finally, the terms with $s^{-1}(i+1) \le p, s^{-1}(i+2) \ge p+1$ annihilate those with $s^{-1}(i+1) \ge p+1, s^{-1}(i+2) \le p$ due to~\ref{I:Fixed} and to the use of the signed shuffle coproduct. Summarizing, one gets $\Csh^{p,q} \boldsymbol{w}\boldsymbol{r}\boldsymbol{v} \in T_p(X;R) \otimes_{\ZZ} \ZZ X^{\times q} + \ZZ X^{\times p} \otimes_{\ZZ} T_q(X;R)$, as announced.

The operation~$\circ$ in Point~3 requires more work. Take two maps $f \in C^p(X; M; R)$, $g \in C^q(X; M; R)$, where this time $p+q = i+j+3$. The co-ideal argument above yields the vanishing of the terms $d^{i+j+2;I_2,I_1}(f) \ast d^{i+j+2;J_1,J_2} (g)(\boldsymbol{w}\boldsymbol{r}\boldsymbol{v})$ of $f \circ g (\boldsymbol{w}\boldsymbol{r}\boldsymbol{v})$, except when $\# I_1 + \# J_1+1$ is $i+1$ or $i+2$. Consider the first case, the second one being similar. Depending on $\# I_1 + \# J_1$ lying in~$I_1$ or in~$J_1$, the term in question evaluates $g$ or~$f$ on an element of $T(X;R)$, and thus vanishes; to see this, use properties \ref{I:Invar}-\ref{I:Invar'} of~$R$, and work with Figure~\ref{P:Circle}. As a result, $f \circ g$ lies in $C^{p+q-1}(X; M; R)$. 

\subsection*{Proof of Proposition~\ref{PR:UnitSubcx}}

Take a critical word of the form $\boldsymbol{w}1\boldsymbol{v} \in Cr_{p+q+1}(X,\sigma)$,  $\boldsymbol{w} \in X^{\times p}$, $\boldsymbol{v} \in X^{\times q}$. The remaining part of our ideal is taken care of by Proposition~\ref{PR:InvarSubcx}. We will show that, up to $\ZZ X \otimes_{\ZZ} T_{p+q}(X,\sigma,1)$, one has $\Csh^{1,p+q}\boldsymbol{w}1\boldsymbol{v} = (-1)^p (1, \boldsymbol{w}\boldsymbol{v})$. Symmetrically, $\Csh^{p+q,1} = (-1)^q  (\boldsymbol{w}\boldsymbol{v}, 1)$ up to $T_{p+q}(X,\sigma,1) \otimes_{\ZZ} \ZZ X$. Since~$1$ acts trivially on all modules, these two terms yield canceling summands in the total differentials. The remaining terms yield summands  lying in $T_{p+q}(X,\sigma,1)$ (with appropriate coefficients on the left and, when necessary, on the right).

Write $\Csh^{1,p+q}\boldsymbol{w}1\boldsymbol{v}$ as $\sum_i (-1)^{i-1}\theta_i  (\boldsymbol{w}1\boldsymbol{v})$, where $\theta_i = b_1\cdots b_{i-1} \in C_{p+q+1}$ acts on $X^{\times(p+q+1)}$ via the braiding~$\sigma$.
\begin{itemize}
 \item For $i \le p$, $\theta_i$ affects only the $\boldsymbol{w}$-part of our word, and thus preserves the letter~$1$. 
 \item For $i=p+2$,  $b_{p+1} (\boldsymbol{w}1\boldsymbol{v}) = \boldsymbol{w}v_1 1 v_2 \ldots v_q$, since $\sigma(1,v_1) = (v_1,1)$ (otherwise the definition of pseudo-unit would imply $\sigma(1,v_1) = (1,v_1)$, which cannot happen in a critical word). This letter~$1$ is not affected by $b_1\cdots b_{p}$, and thus remains in $\theta_{p+2} (\boldsymbol{w}1\boldsymbol{v})$.
 \item For $i>p+2$, write $b_{p+2}\cdots b_{i-1} (\boldsymbol{w}1\boldsymbol{v}) = \boldsymbol{w}1\boldsymbol{v'}$. The subword $v'_1v'_2$ is normal, since the last instance of~$\sigma$ acted at this position. Now, either $\sigma(1,v'_1) = (1,v'_1)$, in which case the normal subword $v'_1v'_2$ survives in $\theta_i (\boldsymbol{w}1\boldsymbol{v})$; or $\sigma(1,v'_1) = (v'_1,1)$, and this letter~$1$ survives in $\theta_i (\boldsymbol{w}1\boldsymbol{v})$.
 \end{itemize} 
In all these cases, $\theta_i (\boldsymbol{w}1\boldsymbol{v})$ lies in $\ZZ X \otimes_{\ZZ} T_{p+q}(X,\sigma,1)$. It remains to analyze $\theta_{p+1}$. Again, criticality implies $\sigma(w_p,1) = (1,w_p)$. The relation $\sigma(w_{p-1},1) = (w_{p-1},1)$ would give the normality of the word $w_{p-1} 1 w_p$, and thus $w_{p-1} w_p$ (recall the definition of pseu\-do-unit), which contradicts the criticality of $\boldsymbol{w}1\boldsymbol{v}$. Thus $\sigma(w_{p-1},1) = (1,w_{p-1})$. This argument iterates until $\sigma(w_{1},1) = (1,w_{1})$, and yields $\theta_{p+1} (\boldsymbol{w}1\boldsymbol{v}) = 1\boldsymbol{w}\boldsymbol{v}$, as desired. 

Next, the words in~$X^*$ containing the letter~$1$ generate a subgroup of $\ZZ\langle X \rangle$ closed under the action of the Coxeter monoids~$C_k$. It is then a co-ideal of $\ZZ\langle X \rangle$. Thus so is $T(X,\sigma,1)$. As a result, $CrC^*(X,\sigma,1; M)$ is a differential graded sub-algebra of $(\, \Map (X^{\times k},M), d^k, \smile\, )$ when $M$ is a bimodule-algebra. It remains to show that in this case the circle product also restricts to $CrC^*(X,\sigma,1; M)$. This is easiest to explain on the example of the term of $f \circ g$ depicted in Figure~\ref{P:Circle}. The crossings where an argument $x_i = 1$ can be ``lost'' are 
\begin{itemize}
 \item those between $I_1$- and $J_2$-strands, and
 \item those between $J_1$- and $I_2$-strands.
\end{itemize} 
Suppose that it happened at the crossing marked in Figure~\ref{P:CircleChase1} (only the relevant part of the diagram is shown). It means that the two strands around this crossing are colored as follows: \raisebox{-.2cm}{\begin{tikzpicture}[scale=0.5]
\draw [rounded corners, myred] (0,0)--(0,0.25)--(1,0.75)--(1,1);
\draw [rounded corners, mygreen] (1,0)--(1,0.25)--(0,0.75)--(0,1);
\node  at (0,0) [left,myred] {$\scriptstyle \boldsymbol{x}$};
\node  at (1,0) [right,mygreen] {$\scriptstyle \boldsymbol{1}$};
\node  at (0,1) [left,mygreen] {$\scriptstyle \boldsymbol{x}$};
\node  at (1,1) [right,myred] {$\scriptstyle \boldsymbol{1}$};
\end{tikzpicture}}. Thus this crossing can be omitted without changing the top colors. Figure~\ref{P:CircleChase1} then proves the relation $\sigma (x''_5,x''_6) = (x''_5,x''_6)$ (recall that $\sigma$ is idempotent). So the map $f \in CrC^5(X,\sigma,1; M)$ is evaluated on a non-critical word, and thus vanishes. 

\begin{figure}[!h]\centering
\begin{tikzpicture}[xscale=0.5,yscale=0.4]
\draw [rounded corners, myred] (3,3.5)--(3,4)--(8.5,8)--(8.5,9);
\draw [rounded corners, myblue] (3,7.5)--(3,9);
\draw [rounded corners, myred] (4,3.5)--(4,4)--(9.5,8)--(9.5,9);
\draw [rounded corners, myblue] (4,7.5)--(4,9);
\draw [rounded corners] (5,3.5)--(5,9);
\draw [rounded corners, mygreen] (8,3.5)--(8,4.5)--(6,7.5)--(6,9);
\draw [rounded corners, mygreen,shift={(1,0)}] (8,3.5)--(8,4.5)--(6,7.5)--(6,9);
\draw [rounded corners, myviolet] (6,3.5)--(6,4.5)--(3,6);
\draw [dotted, thick] (2.5,4)--(6.5,4);
\draw [dotted, thick] (2.5,8)--(7.5,8);
\draw [red, thick] (6.9,6) circle [radius=0.4];
\node at (2.5,4.4) {$\scriptstyle g$};
\node at (2.3,8.3) {$\scriptstyle f$};
\node at (8.5,9) [above, myred]{$\scriptstyle x''_1$};
\node at (3,9) [above, myblue]{$\scriptstyle x''_2$};
\node at (9.5,9) [above, myred]{$\scriptstyle x''_3$};
\node at (4,9) [above, myblue]{$\scriptstyle x''_4$};
\node at (5,9) [above]{$\scriptstyle x''_5$};
\node at (6,9) [above, mygreen]{$\scriptstyle x''_6$};
\node at (7,9) [above, mygreen]{$\scriptstyle x''_7$};
\node at (3,3.5) [below, myred]{$\scriptstyle I_1$};
\node at (4,3.5) [below, myred]{$\scriptstyle I_1$};
\node at (9,3.5) [below, mygreen]{$\scriptstyle J_2$};
\node at (8,3.5) [below, mygreen]{$\scriptstyle J_2$};
\node at (6,3.5) [below, myviolet]{$\scriptstyle I_2$};
\node at (11,6) {$\;\leadsto\;$};
\end{tikzpicture} 
\begin{tikzpicture}[xscale=0.5,yscale=0.4]
\draw [rounded corners, red, thick, dashed] (4,3.5)--(4,4)--(6.8,5.8)--(6.8,6.3)--(6,7.5)--(6,9);
\draw [rounded corners, red, thick, dashed] (8,3.5)--(8,4.5)--(7.2,5.8)--(7.2,6.3)--(9.5,8)--(9.5,9);
\draw [rounded corners, myred] (3,3.5)--(3,4)--(8.5,8)--(8.5,9);
\draw [rounded corners, myblue] (3,7.5)--(3,9);
\draw [rounded corners, myblue] (4,7.5)--(4,9);
\draw [rounded corners] (5,3.5)--(5,9);
\draw [rounded corners, mygreen,shift={(1,0)}] (8,3.5)--(8,4.5)--(6,7.5)--(6,9);
\draw [rounded corners, myviolet] (6,3.5)--(6,4.5)--(3,6);
\draw [dotted, thick] (2.5,4)--(6.5,4);
\draw [dotted, thick] (2.5,8)--(7.5,8);
\node at (2.5,4.4) {$\scriptstyle g$};
\node at (2.3,8.3) {$\scriptstyle f$};
\node at (8.5,9) [above, myred]{$\scriptstyle x''_1$};
\node at (3,9) [above, myblue]{$\scriptstyle x''_2$};
\node at (9.5,9) [above, myred]{$\scriptstyle x''_3$};
\node at (4,9) [above, myblue]{$\scriptstyle x''_4$};
\node at (5,9) [above]{$\scriptstyle x''_5$};
\node at (6,9) [above, mygreen]{$\scriptstyle x''_6$};
\node at (7,9) [above, mygreen]{$\scriptstyle x''_7$};
\node at (3,3.5) [below, myred]{$\scriptstyle I_1$};
\node at (4,3.5) [below, myred]{$\scriptstyle I_1$};
\node at (9,3.5) [below, mygreen]{$\scriptstyle J_2$};
\node at (8,3.5) [below, mygreen]{$\scriptstyle J_2$};
\node at (6,3.5) [below, myviolet]{$\scriptstyle I_2$};
\node at (11,6) {$\;\overset{\text{R}\mathrm{III}}{\leadsto}\;$};
\end{tikzpicture} 
\begin{tikzpicture}[xscale=0.5,yscale=0.4]
\draw [rounded corners, red, thick, dashed] (4,3.5)--(4,6)--(6,7)--(6,9);
\draw [rounded corners, red, thick, dashed] (8,3.5)--(8,4.5)--(7.2,5.8)--(7.2,6.3)--(9.5,8)--(9.5,9);
\draw [rounded corners, myred] (3,3.5)--(3,4)--(8.5,8)--(8.5,9);
\draw [rounded corners, myblue] (3,7.5)--(3,9);
\draw [rounded corners, myblue] (4,7.5)--(4,9);
\draw [rounded corners] (5,3.5)--(5,9);
\draw [rounded corners, mygreen,shift={(1,0)}] (8,3.5)--(8,4.5)--(6,7.5)--(6,9);
\draw [rounded corners, myviolet] (6,3.5)--(6,4.5)--(3,6);
\draw [dotted, thick] (2.5,4)--(6.5,4);
\draw [dotted, thick] (2.5,8)--(7.5,8);
\node at (2.5,4.4) {$\scriptstyle g$};
\node at (2.3,8.3) {$\scriptstyle f$};
\node at (8.5,9) [above, myred]{$\scriptstyle x''_1$};
\node at (3,9) [above, myblue]{$\scriptstyle x''_2$};
\node at (9.5,9) [above, myred]{$\scriptstyle x''_3$};
\node at (4,9) [above, myblue]{$\scriptstyle x''_4$};
\node at (5,9) [above]{$\scriptstyle x''_5$};
\node at (6,9) [above, mygreen]{$\scriptstyle x''_6$};
\node at (7,9) [above, mygreen]{$\scriptstyle x''_7$};
\node at (3,3.5) [below, myred]{$\scriptstyle I_1$};
\node at (4,3.5) [below, myred]{$\scriptstyle I_1$};
\node at (9,3.5) [below, mygreen]{$\scriptstyle J_2$};
\node at (8,3.5) [below, mygreen]{$\scriptstyle J_2$};
\node at (6,3.5) [below, myviolet]{$\scriptstyle I_2$};
\end{tikzpicture} 
\caption{If in a term of $f \circ g$ an $I_1$-$J_2$ crossing has no effect on colors, then this term evaluates~$f$ on a critical word.}\label{P:CircleChase1}
\end{figure}
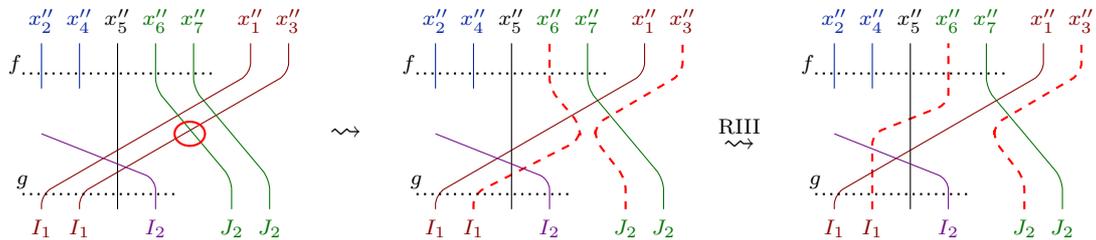

\subsection*{Proof of Theorem~\ref{T:TotalSh}}

\begin{enumerate}
\item The quantum symmetrizer can be seen as the action of 
$$Sh_k:= \sum_{s \in S_k} (-1)^{|s|} T_s \; \in \ZZ B_k^{+}$$ 
on $\ZZ X^{\times k}$ via~$\sigma$. For all $1\le i < k$, this linear combination decomposes in two ways:
\begin{align}\label{E:TotShLocal}
Sh_k &= (1-b_i)t_i = t'_i(1-b_i) &\text{ for some } t_i,t'_i \in \ZZ B_k^{+}.
\end{align} 
The first decomposition yields $\mu_i \TotSh_k = 0$, where $\mu_i$ multiplies the factors $i$ and $i+1$ of $\Mon(X,\sigma)^{\times k}$. From this one deduces that the $Sh_k$ induce morphisms between the announced (co)\-chain complexes. The second decomposition implies that $\TotSh_k$ vanishes on $k$-tuples with $\sigma(x_i,x_{i+1}) = (x_i,x_{i+1})$.

\item In the pseudo-unital case, the $\oTotSh_k$ yield morphisms of complexes since they are compositions of three components that do so. Further, if $x_i = 1$ for some $i$, then the word $T_s(x_1 \ldots x_k)$ still contains a letter~$1$ for all $s \in S_k$. Thus the reduced quantum symmetrizer vanishes on the $\ZZ\langle X \rangle 1 \ZZ\langle X \rangle$ part of $T_k(X,\sigma,1)$. The $\ZZ\langle X \rangle R_- \ZZ\langle X \rangle$ part is taken care of by the $\TotSh_k$ component of~$\oTotSh_k$.

\item It is a consequence of the following decomposition of the quantum symmetrizer:
$$\TotSh_{p+q} = (\TotSh_p \otimes \TotSh_q) \Csh^{p,q}.$$
\end{enumerate}

\subsection*{Proof of Theorem~\ref{T:BrHomIdempot}}

Our proof is based on the machinery of the algebraic discrete Morse theory in the context of Anick resolutions. For more details see the original papers \cite{Anick,Skol,JoWe} or a recent survey in~\cite{Lopatkin}.

We will work only in the pseudo-unital setting, since an idempotent braiding~$\sigma$ on~$X$, extended to $X^+= X\sqcup \{1\}$ by $\sigma(x,1) = \sigma(1,x) = (1,x)$ for all~$x$, defines a PUIBS such that the monoid $\oMon(X^+,\sigma,1)$ recovers $\Mon(X,\sigma)$, being a module over $(X,\sigma)$ is equivalent to being a module over $(X^+,\sigma,1)$, and the critical complexes for $(X^+,\sigma,1)$ and for $(X,\sigma)$ coincide.

According to Proposition~\ref{PR:PUIBS}, the reduced structure monoid $\oMon(X,\sigma,1)$ can be replaced with the monoid $N := (\oNorm(X,\sigma), \ast, \emptyw)$ (Notation~\ref{N:PUIBSast}). Further, it suffices to work with the bar differential on $\ZZ N \times N^* \times N$ and with the critical complexes $\ZZ N \times Cr_*(X,\sigma,1) \times N$. That is, we take~$N$ as two-sided coefficients (Example~\ref{EX:StructureSGAsBimodAlg}). Applying the functors $M \otimes_{N \times N^{op}} -$\, or $\Hom_{N \times N^{op}}(-, M)$, one gets the (co)chain complexes from the announced quasi-isomorphisms.  

Define a weighted oriented graph~$\Gamma$ as follows. Its vertex set is $N \times \overline{N}^* \times N$, with $\overline{N} = N \setminus \{ \emptyw \}$. The vertices from $N \times \overline{N}^{\times k} \times N$ are said to have \emph{degree}~$k$. The edges encode the multiplication of neighbouring words---i.e., connect $(\boldsymbol{w}_0, \ldots, \boldsymbol{w}_{k+1})$ to $(\boldsymbol{w}_0, \ldots, \boldsymbol{w}_i \ast \boldsymbol{w}_{i+1},\ldots, \boldsymbol{w}_{k+1})$ for all $0 \le i \le k$, whenever the latter tuple lies in $N \times \overline{N}^{\times(k-1)} \times N$. Such an edge $e$ is declared to be of \emph{type} $i$ and \emph{weight} $w(e) = (-1)^i$. Edges of type $0 < i <k$ are called \emph{internal}. This graph models the normalized bar complex: the bar differential is computed by the formula 
 $$d^{bar}_k(v) = \sum_{e \colon v \to v'} w(e) v'.$$ Now, reverse in~$\Gamma$ all internal edges 
 \begin{align}\label{E:Reversed1}
& (\boldsymbol{w}_0,x_1, \ldots, x_i, \boldsymbol{w}_{i+1},\ldots, \boldsymbol{w}_{k+1}) \to (\boldsymbol{w}_0,x_1, \ldots, x_i\boldsymbol{w}_{i+1},\ldots, \boldsymbol{w}_{k+1}) 
 \end{align}
 such that all $x_s \in X$, all $\boldsymbol{w}_t \in N$, $\boldsymbol{w}_t \neq \emptyw$ for $t \neq 0,k+1$, the word $x_1 \ldots x_i$ is critical, and the word $x_i\boldsymbol{w}_{i+1}$ is normal (i.e., equals $x_i\ast \boldsymbol{w}_{i+1}$). Change the weight of all reversed edges to the opposite. Denote this new graph by~$\Gamma'$. 
  
Let us check that in~$\Gamma'$ any vertex $(\boldsymbol{w}_0, \ldots, \boldsymbol{w}_{k+1})$ is the beginning of a finite number of paths only. Consider the total concatenated word $\boldsymbol{w} = \boldsymbol{w}_0 \ldots \boldsymbol{w}_{k+1}$. Let it contain $l$ letters. Moving along the edges of~$\Gamma'$, one can change this amalgamated word only by the action of an element of the Coxeter monoid~$C_l$ (possibly followed by erasing some $1$s---in which case one can apply induction on~$l$). Further, if during this process one can reach a word $b (\boldsymbol{w})$ from $c (\boldsymbol{w})$, with $b,c \in C_l$, then~$c$ right divides~$b$. Lemma~\ref{L:CoxMon} guarantees that the relation ``right divisor'' is a partial order on the finite set~$C_l$. Thus it suffices to show that $(\boldsymbol{w}_0, \ldots, \boldsymbol{w}_{k+1})$ starts only a finite number of \emph{stable} paths---i.e., in which the total word is preserved. A stable path can have at most $l-1$ reversed edges. Indeed, after a reversed edge of the form~\eqref{E:Reversed1}, one cannot follow a new reversed edge until $\boldsymbol{w}_0$ absorbs all the letters $x_1, \ldots, x_i$ (and no letters can ever leave $\boldsymbol{w}_0$). Looking at how the degrees of the vertices behave, one concludes that the length of a stable path from $(\boldsymbol{w}_0, \ldots, \boldsymbol{w}_{k+1})$ is at most $k+2(l-1)$. Since any vertex starts a finite number of edges, we are done.

Further, any vertex belongs to at most one reversed edge. The vertices disjoint from the reversed edges are precisely those of the form $(\boldsymbol{w}_0,x_1, \ldots, x_k, \boldsymbol{w}_{k+1}) \in N \times Cr_k(X,\sigma,1) \times N =:CrV_k$. They are called \emph{critical}.
 
 Define $d^{cr}_k \colon \ZZ CrV_k \to \ZZ CrV_{k-1}$ as the linearization of 
 $$d^{cr}_k(v) = \sum_{v' \in CrV_{k-1}} (\sum_{p \in P'(v,v')}w(p))v',$$ where $v \in CrV_k$, $P'(v,v')$ denotes the set of oriented paths from~$v$ to~$v'$ in~$\Gamma'$, and the weight $w(p)$ of a path $p$ is the product of the weights of its edges. The algebraic discrete Morse theory tells us that the $d^{cr}_k$ define a differential, and that the maps
\begin{align}
(\ZZ N \times \overline{N}^{\times k} \times N,d^{bar}_k) &\,\longleftrightarrow\, (\ZZ CrV_k,d^{cr}_k),\notag\\
N \times \overline{N}^{\times k} \times N \ni\; v &\,\overset{\pi}{\longmapsto}\, \sum_{v' \in CrV_{k}} (\sum_{p \in P'(v,v')}w(p))v',\label{E:pi}\\
\sum_{v \in N \times \overline{N}^{\times k} \times N} (\sum_{p \in P'(v',v)}w(p))v &\,\overset{\iota}{\longmapsfrom}\, v' \; \in CrV_{k}\label{E:iota}
\end{align}
 yield a quasi-isomorphism of complexes. It remains to identify the complex on the right with the one defining the critical braided homology of $(X,\sigma,1)$ with coefficients in~$N$. Since both complexes share the same underlying graded abelian group, it suffices to compare $d^{cr}_k$ with the critical version of the braided differential. 
 
 It is thus essential to understand paths $p \in P'(v,v')$ for given $v \in CrV_{k}$, $v' \in CrV_{k-1}$. Such a path has to alternate non-reversed edges with reversed ones, since the reversed edges are pairwise disjoint and terminate in non-critical vertices. A vertex $(\boldsymbol{w}_0, \ldots, \boldsymbol{w}_{k+1})$ where the internal words $\boldsymbol{w}_1, \ldots, \boldsymbol{w}_k$ are one-letter words is never a source of a reversed edge. A vertex whose internal words are one-letter except for one two-letter word $xy \in N$ can only be a source of a reversed edge that splits $xy$ into two words $x,y$. Thus $p$ begins with a sequence of internal two-edge segments of the form 
\begin{align}\label{E:ZigZagPath}
(\ldots, x_i, x_{i+1}, \ldots) \to (\ldots, x_i \ast x_{i+1} = x'_{i+1} x'_i, \ldots) \to (\ldots, x'_{i+1}, x'_i, \ldots),
\end{align}
 where all the $x$s are letters, and $(x'_{i+1}, x'_i) = \sigma(x_i, x_{i+1}) \neq (x_i, x_{i+1})$. This is precisely the action of the generator~$b_i$ of~$C_k$ via~$\sigma$. The weight of each such two-edge segment is $-1$. The remainder of~$p$ is a single edge of type~$i$ and weight $(-1)^i$. We study separately three possibilities for~$i$:
\begin{enumerate}
 \item $i=0$. In this case, the two-edge segments should be of types $j-1, j-2, \ldots, 1$, in this order, for some $1 \le j \le k$. Otherwise, the overall effect of the two-edge segments would be the action of $b b_1 \cdots b_{j-1} \in C_k$, with a non-trivial~$b$ not containing~$b_1$. The action of~$b$ would then prevent the final vertex~$v'$ from being critical.
 \item $i=k$. Similarly, the two-edge segments should be of types $j, j+1, \ldots, k-1$.
 \item $0<i<k$. This means that our path finishes with an internal edge 
$$(\boldsymbol{w}_0, x_1,\ldots) \to (\boldsymbol{w}_0, x_1,\ldots, x_{i-1},x_i \ast x_{i+1}, x_{i+1}, \ldots),$$ 
where the product $y = x_i \ast x_{i+1}$ consists of one letter. Assume that $\sigma(x_i, x_{i+1}) = (1,y)$, the case $(y,1)$ being symmetric. An argument similar to that used for type~$0$ shows that the two-edge segments should be of types $j-1, j-2, \ldots, i+1$, in this order, for some $i < j \le k$. An argument similar to that from the proof of Proposition~\ref{PR:UnitSubcx} yields $b_1 \cdots b_{i-1}(x_1\ldots x_{i-1} 1) = 1 x_1\ldots x_{i-1}$.
\end{enumerate}  
Combining the three cases, one recovers all the components $d^l_{k;j}$ and $d^r_{k;j}$ of the braided differentials, with the correct signs. They are interpreted differently according to whether or not the letter~$1$ appears in their step-by-step computation. It remains to check that all the above paths are valid, provided that their final vertex is critical. Two issues can occur.
\begin{enumerate}[label=\alph*)]
\item\label{I:PB1} One does not necessarily have a reversed edge 
$$(\ldots, x_{i-1},x'_{i+1} x'_i, \ldots) \to (\ldots, x_{i-1},x'_{i+1}, x'_i, \ldots).$$ 
In the case of a path corresponding to $d^l$, this happens only if $i \ge 2$ and $x_{i-1}x'_{i+1}x'_i$ is a normal word. But then the action of~$b_{i-1}$ on $\ldots x_{i-1} x'_{i+1}  x'_i \ldots$ is trivial, and the~$b_{u}$ with $u < i-1$ do not touch the normal subword $x'_{i+1} x'_i$. Then the final vertex cannot be critical. The $d^r$ case is analogous.
\item\label{I:PB2} The edge $(\ldots, x_i, x_{i+1}, \ldots) \to (\ldots, x_i \ast x_{i+1}, \ldots)$ might be unavailable since it was reversed in~$\Gamma'$. The word $x_ix_{i+1}$ is then normal. An argument analogous to Point~1 shows that the final vertex is then not critical.
\end{enumerate}

To conclude, one should identify the map~\eqref{E:iota} with the quantum symmetrizer. This is done by a study of paths $p \in P'(v',v)$ for a critical vertex~$v'$ and an arbitrary vertex~$v$ of the same degree~$k$. Repeating the arguments above, one sees that such a path consists of two-edge segments~\eqref{E:ZigZagPath} for $i$ taking values 
$$1, \ldots, i_{2}; \ldots; k-2, \ldots, i_{k-1}; k-1, \ldots, i_{k},$$
in this order, for some $i_2 \le 2, \ldots, i_k \le k$; here $j-1,j$ is considered to be an empty sequence. Such a path transforms $x_1 \ldots x_k$ into $T_s(x_1 \ldots x_k)$, where
$$s = s_1 \cdots s_{i_2} \cdots\cdots s_{k-2} \cdots s_{i_{k-1}} s_{k-1} \cdots s_{i_k}.$$
Every permutation~$s$ has a unique decomposition of this type. Looking at the weights of such paths, one concludes that the map~\eqref{E:iota} coincides with~$\TotSh_k$, modulo the terms corresponding to the paths which are unavailable for the reasons of types~\ref{I:PB1} and~\ref{I:PB2} above. Repeating the above treatment of these two issues, one sees that for such a path a normalized pair appears to the left of some~$j$ after two-edge segments $j-1, \ldots, i_{j}$. So the first $j-1$ letters of our vertex form a critical word at this stage. For each unavailable path, choose the maximal~$j$ with this property. Now, the terms of $\TotSh_k$ corresponding to unavailable paths with the same value of~$j$ and the same part $j-1, \ldots, i_{j}; \ldots; k-1, \ldots, i_{k}$  sum up to
$$\pm (\TotSh_{j-1} \otimes \Id)b_{j-1} \cdots b_{i_{j}} \cdots \cdots b_{k-1} \cdots b_{i_{k}}.$$
Since $\TotSh_{j-1}$ vanishes on all critical words, we are done.

\bigskip
\addcontentsline{toc}{section}{References}
\footnotesize
\bibliographystyle{alpha}
\bibliography{biblio}
\end{document}